\documentclass{siamart251216}
\usepackage[utf8]{inputenc}
\usepackage{csquotes}
\usepackage{amsfonts}
\usepackage{amsmath,amssymb,accents,empheq} 
\usepackage{enumitem}
\usepackage{subcaption}
\usepackage{graphicx}
\usepackage{bm}
\usepackage{float}
\usepackage{xcolor}
\usepackage{hyperref}
\usepackage{cite}

\usepackage{cleveref}
\usepackage{ulem}

\usepackage{tikz}

\usepackage{algorithm}
\usepackage{algpseudocode}
\numberwithin{equation}{section}

\hypersetup{
    colorlinks=true,
    linkcolor=green,
    citecolor=green,
    filecolor=green,      
    urlcolor=green,
}

\usepackage{listings}
 
\definecolor{codegreen}{rgb}{0,0.6,0}
\definecolor{codegray}{rgb}{0.5,0.5,0.5}
\definecolor{codepurple}{rgb}{0.58,0,0.82}
\definecolor{backcolour}{rgb}{0.95,0.95,0.92}

\title{A Structure-preserving Adaptive-Rank Approach to the High-Dimensional Wigner-Poisson System}
\author{Andrew J. Christlieb 
\thanks{Department of Computational Mathematics, Science and Engineering, Michigan State University, East Lansing, MI, 48824}
\and 
Sining Gong \thanks{Corresponding author. Department of Computational Mathematics, Science and Engineering, Michigan State University, East Lansing, MI, 48824 (\email{gongsini@msu.edu})}
\and 
Jing-Mei Qiu \thanks{Department of Mathematical Sciences, University of Delaware, Newark, DE, 19716} \and 
Nanyi Zheng \footnotemark[3]}

\headers{Adaptive-Rank Approach for Wigner-Poisson}{A. Christlieb, S. Gong, J. Qiu, and N. Zheng}

\begin{document}

\maketitle
\begin{abstract}
The Wigner–Poisson system is a deterministic phase-space model for quantum kinetic electron dynamics, but high-dimensional simulations are limited by the full 3D3V phase space and the nonlocal Wigner potential. We develop a structure-preserving, sampling-based adaptive-rank solver in hierarchical Tucker format for finite-$H$ regimes in which Wigner–Poisson solutions exhibit exploitable low-rank structure. The central difficulty is that adaptive compression can destroy the Fourier-Hermitian tensor symmetry required for a real inverse velocity transform and can break discrete global conservation laws. We address these issues with a Fourier-Hermitian-symmetry-aware sampling and mapping procedure and a global moment correction enforcing mass, momentum, and self-consistent total energy. Numerical tests for two-stream instability and strong Landau damping in 2D2V and 3D3V show roundoff-level conservation, preservation of the real-valued inverse transform, and approximately linear scaling with respect to the number of grid points per coordinate over the tested rank range. The results demonstrate that long-time 3D3V Wigner–Poisson simulations can be performed without assembling the full phase-space tensor.

\end{abstract}



\section{Introduction}
The high-dimensional Wigner–Poisson (WP) system accounts for quantum wave-packet spreading, diffraction, and uncertainty effects, making it a natural model for electron dynamics in warm dense matter (WDM). In the WDM regimes relevant to $\alpha$-particle stopping-power studies at the National Ignition Facility (NIF)\cite{wigner1932quantum,bonitz2016quantum,graziani2014kinetic}, electrons are hot enough that Pauli exclusion effects may be neglected but dense enough that quantum effects remain important; their collective mean-field behavior is therefore well described by the WP system \cite{zylstra2022burning}.
However, simulating such a system presents a severe computational challenge due to its non-local structure, particularly in higher dimensions. Traditional methods for parallel discretization quickly become intractable due to prohibitive communication costs at each time step. 
Our previous work \cite{christlieb2025sampling} introduced a sampling-based adaptive-rank solver for the 1D1V WP system. In this work, we extend that approach to high-dimensional phase-space settings and incorporate a correction that enforces global conservation of mass, momentum, and energy.
To the best of our knowledge, this work achieves the first computationally tractable, long-time simulations of the WP system in 3D3V.

Deterministic mesh-based methods for the Wigner and Wigner–Poisson equations have been studied extensively since the mid-1980s. Early work developed finite-difference and integral-sum discretizations for quantum transport and resonant-tunneling devices, including the foundational schemes of Frensley \cite{frensley1987wigner,frensley1988quantum}, the central-difference and self-consistent WP formulations of Jensen and Buot \cite{jensen1989numerical,buot1990lattice,jensen1991methodology}, the difference-spectral methods of Ringhofer \cite{ringhofer1990spectral,ringhofer1992spectral}, and the reformulation of Mains and Haddad \cite{mains1994accurate}. Related analytical and numerical foundations for semi-discrete and discrete-velocity Wigner equations were developed in \cite{goudon2002analysis,arnold2000discrete}. More recent deterministic methods include adaptive conservative cell-average spectral element discretizations \cite{shao2011adaptive}, improved boundary treatments \cite{jiang2011accuracy}, momentum-domain narrowing \cite{kim2016efficient}, schemes for unbounded potentials \cite{chen2019numerical}, hybrid discretizations \cite{jiang2022hybrid,sun2024hybrid}, and explicit spectral treatments of the Wigner term subject to CFL restrictions \cite{xiong2016advective}.

Because of the nonlocal pseudo-differential structure of the Wigner operator, splitting methods are often used to isolate transport and potential-update subproblems \cite{suh1991numerical,arnold1995operator,ArnoldRinghofer1996,furtmaier2016semi,dorda2015weno,chen2019high}. Semi-Lagrangian and analytical spectral updates can reduce or remove transport CFL restrictions, while stability analyses and high-order spectral-WENO variants provide accurate deterministic solvers for the split subproblems \cite{suh1991numerical,arnold1995operator,ArnoldRinghofer1996,dorda2015weno,chen2019high,chen2022higher}. In parallel with these deterministic approaches, stochastic and Monte Carlo Wigner methods, including signed-particle simulations, branching random walks, and tunneling frameworks, have also been developed \cite{shifren2001particle,nedjalkov2004unified,sellier2014benchmark,shao2015comparison,shao2016computable,muscato2016class}. Broader reviews of deterministic and stochastic Wigner-function methods can be found in \cite{sellier2015introduction,weinbub2018recent}.

Beyond collisionless WP models, Wigner–Poisson–BGK equations incorporate relaxation effects for energy dissipation, charge continuity, and diffusion limits. Their asymptotic behavior, drift-diffusion limits, well-posedness, and Cauchy theory have been analyzed in \cite{arnold1996self,manzini2006rigorous,manzini2010diffusive,li2016cauchy}. Such collisional models have been used to simulate self-sustained current oscillations in semiconductor superlattices \cite{bonilla2005wigner,alvaro2010two} and have motivated asymptotic-preserving schemes and relaxation-time approximations \cite{crouseilles2014asymptotic,fernandes2015relaxation}. Although collisions are outside the scope of the present work, the adaptive-rank and conservation framework developed here is intended to be extensible to such settings.

Several classes of methods have been developed to mitigate the curse of dimensionality in high-dimensional PDEs, including sparse grids \cite{BungartzGriebel2004}, separated representations \cite{BeylkinMohlenkamp2005}, tensor-train and hierarchical tensor formats \cite{Oseledets2011,Hackbusch2012}, and modern low-rank tensor PDE solvers and surveys \cite{Khoromskij2012,Bachmayr2023}. These methods motivate the use of low-rank structure, but WP introduces additional constraints because the nonlocal Wigner operator is most efficient in Fourier variables and the inverse transform must remain real-valued. Low-rank truncation can destroy conservation properties even when the underlying full-grid discretization is conservative. This issue has been addressed in dynamical low-rank Vlasov solvers, including projector-splitting and quasi-conservative methods \cite{EinkemmerLubich2018,EinkemmerLubich2019}, mass/momentum/energy-conservative dynamical low-rank schemes \cite{Einkemmer2021}, conservative low-rank tensor methods \cite{GuoQiu2024Conservative}, and local macroscopic conservative low-rank tensor methods \cite{GuoQiu2024LoMaC}. The present correction is closer in spirit to moment-matching and conservative projection, but it is adapted to the self-consistent WP total energy and the hierarchical Tucker adaptive cross
approximation (HTACA) sampling framework.

In this paper, we extend the sampling-based adaptive-rank solver developed in \cite{christlieb2025sampling} from the 1D1V WP system to high-dimensional phase-space settings by incorporating HTACA. To the best of our knowledge, this is the first structure-preserving adaptive-rank solver for high-dimensional WP systems. The
main contributions are threefold. First, we design an adaptive-rank high-dimensional WP solver that avoids assembling the full phase-space tensor and exhibits nearly linear scaling with respect to the number of grid points per coordinate under the tested truncation tolerance for an arbitrary CFL condition. Second, we introduce a symmetry-aware sampling and mapping strategy in velocity Fourier space, ensuring that the inverse Fourier transform remains real-valued up to roundoff. Third, inspired by the LoMaC framework in \cite{GuoQiu2024LoMaC} and using adaptive weight functions constructed from leading singular vectors in velocity space \cite{ZSQ2026}, we develop a global moment-correction procedure for the adaptive-rank solution. The correction enforces discrete global conservation of mass, all $d$
momentum components, and self-consistent total energy through a rank-$(d+2)$
update, where $d$ is the physical dimension of the WP system. Numerical tests in
2D2V and 3D3V demonstrate that the proposed method captures high-dimensional Wigner dynamics while preserving the desired symmetry and conservation properties, and exhibiting approximately linear runtime growth in the tested benchmark examples for the ranks induced by the prescribed truncation tolerances.

The remainder of the paper is organized as follows. Section~\ref{sec:WP}
introduces the nondimensional WP system, its moment equations, and the
improved full-rank Strang-splitting solver used as the reference discretization.
Section~\ref{sec:highD_WP_solver} presents the adaptive-rank solver, including
the HTACA representation, the Hermitian-symmetry-preserving Fourier update, and
the global conservation correction. Section~\ref{sec:numerical} reports
high-dimensional numerical tests for two-stream instability and strong Landau
damping, including conservation, symmetry, and scaling results. Finally,
Section~\ref{sec:conclusion} summarizes the findings and discusses future
extensions.

\section{Wigner--Poisson System}\label{sec:WP}
In this section, we introduce the high-dimensional WP system with its nondimensional formulation, the continuous quantum fluid model leading to the conservation laws, and the improved full-rank solver with structure-preserving properties that the adaptive-rank algorithm will emulate. 

\subsection{Model Setup}\label{subsec:Wigner poisson system}

The WP system in high dimensions, as presented in~\cite{arnold1995operator,ArnoldRinghofer1996,markowich2012semiconductor}, takes the form:
\begin{subequations}\label{eqn:WPmodel}
\begin{align}
\frac{\partial \tilde{f}}{\partial \tilde{t}}
+ \tilde{v}\cdot \nabla_s \tilde{f}
&=
\begin{aligned}[t]
&\frac{-i e m_e^d}{(2\pi)^d \hbar^{d+1}}
\int_{\mathbb{R}^d_{\tilde{v}'}}\int_{\mathbb{R}^d_{s'}}
d\tilde{v}'\, ds'\,
\exp \left(\frac{i m_e (\tilde{v}'-\tilde{v})\cdot s'}{\hbar}  \right)\\
&\qquad \times
\left[
\phi\left(s+\frac{s'}{2}\right)-\phi\left(s-\frac{s'}{2}\right)
\right]
\tilde{f}(s,\tilde{v}',\tilde{t}),
\end{aligned}
\label{eqn:WPmodel_wigner}
\\
\Delta_s \phi
&=
-\frac{e}{\epsilon_0}
\left(
\int_{\mathbb{R}^d_{\tilde v}} d\tilde{v}\,
\tilde{f}
-
n_0
\right).
\label{eqn:WPmodel_poisson}
\end{align}
\end{subequations}
where $s, \tilde{v} \in \mathbb{R}^d \, (d = 1,2, \text{or } 3)$, $\tilde{f}(s, \tilde{v}, \tilde{t})$ is the Wigner distribution function in phase space, and $\phi(s)$ denotes the electrostatic potential. The physical constants are $m_e$, the particle mass; $e$, the elementary charge; $\epsilon_0$, the vacuum permittivity; $n_0$, the background number density; $\hbar$, the reduced Planck constant; and $\tilde{t}$, the physical time. 
In the Wigner equation, the potential acts on the distribution function via the right-hand side, which constitutes a highly non-local pseudo-differential operator. Unlike the classical Vlasov–Poisson (VP) model, where the potential acts through a simple local derivative, the quantum potential depends on a spatial displacement $s'$ and involves a highly oscillatory phase integral over the entire velocity space. Notably, under the correct scaling, it can be shown that the WP model asymptotically recovers the local classical VP formulation. For a rigorous proof of this reduction in the 1D1V nondimensional case, we refer the reader to the appendix of \cite{christlieb2025sampling}; the extension to high-dimensional cases follows as a straightforward generalization. This limit is used only for interpretation; the numerical method is designed for finite $H$ introduced later in \eqref{eqn:nonD-WP-system}.

We now nondimensionalize the WP system to facilitate scaling analysis and numerical discretization. Let $\tau$, $l$, and $\bar{\phi}$ denote the characteristic time, length, and electrostatic potential scales, respectively. The corresponding nondimensional variables are defined as:
\[
\tilde{t} = \tau t,\quad s = l \bm{x},\quad \phi = \bar{\phi} \Phi,\quad \tilde{v} = \frac{l}{\tau} {\bm{v}},\quad f = \frac{l^d}{n_0 \tau} \tilde{f}; \quad  ds = l^d d\bm{x}, \, d\tilde{v} = \frac{l^d}{\tau^d} \, d{\bm{v}}.
\]
Substituting these into the WP system \eqref{eqn:WPmodel} yields:
\begin{subequations}
\begin{align}
\frac{\partial f}{\partial t} + {\bm{v}} \cdot \nabla_{\bm{x}} f &= -\frac{iC}{(2\pi)^d H^{d+1}} \int_{\mathbb{R}^d_{{\bm{v}}'}}\int_{\mathbb{R}^d_{\bm{x}'}} d{\bm{v}}' d{\bm{x}}' \, \exp\left(i \frac{{\bm{v}}' - {\bm{v}}}{H} \cdot {\bm{x}}'\right) \\
& \qquad \times \left[\Phi\left({\bm{x}} + \frac{{\bm{x}}'}{2}\right) - \Phi\left({\bm{x}} - \frac{{\bm{x}}'}{2}\right)\right] f({\bm{x}}, {\bm{v}}', t), \nonumber \\
\Delta_{\bm{x}} \Phi &= -D \left(\int_{\mathbb{R}^d_{{\bm{v}}}} f \, d{\bm{v}} - \rho_0 \right), \qquad \text{with } \rho_0 = 1,
\end{align}
\end{subequations}
where $\bm{x},\bm{v} \in \mathbb{R}^d$. 
The resulting dimensionless parameters are given by:
$$C = \frac{e \bar{\phi} \tau^2 }{m_e l^{2}},\quad H = \frac{\tau \hbar}{m_e l^2},\quad D = \frac{e n_0 l^2}{\bar{\phi} \epsilon_0}.$$
To simplify the system and highlight key physical scalings, we choose the potential scale as $\bar{\phi} = \frac{e n_0 l^2}{\epsilon_0}$. We define the time and length scales based on characteristic plasma parameters: time is scaled by the plasma frequency $\omega_{pe}$, and space by the Debye length $\lambda_D$, given respectively by
$$\omega_{pe} = \sqrt{ \frac{e^2 n_0}{m_e \epsilon_0} }, \qquad \lambda_D = \sqrt{ \frac{ \epsilon_0 k_B T_0}{n_0 e^2} },$$
where $T_0$ denotes the reference temperature. Under these scalings, the velocity scale naturally becomes the thermal velocity, a standard metric in kinetic plasma descriptions:
$$\frac{l}{\tau} = \lambda_D \omega_{pe} = \sqrt{ \frac{ \epsilon_0 k_B T_0}{n_0 e^2} } \cdot \sqrt{ \frac{e^2 n_0}{m_e \epsilon_0} } = \sqrt{ \frac{k_B T_0}{m_e} } = v_{th}.$$
As a result of these choices, $C$ and $D$ both reduce to unity, leaving $H$ as the sole remaining parameter quantifying the strength of quantum effects in the dimensionless formulation:
\begin{equation}\label{eqn:quantum parameter}
C = \frac{e \bar{\phi} }{m_e l^2 \omega_{pe}^2} 
= \frac{e}{m_e l^2} \cdot \frac{e n_0 l^2}{\epsilon_0} \cdot \frac{m_e \epsilon_0}{e^2 n_0} = 1, \quad
H = \frac{ \hbar}{m_e \lambda_D^2 \omega_{pe}}, \quad
D = \frac{e n_0 l^2}{\bar{\phi} \epsilon_0} = 1.
\end{equation}
Thus, the finalized nondimensional WP system takes the form:
\begin{subequations}\label{eqn:nonD-WP-system}
\begin{align}
\label{eq:wigner}
\frac{\partial f}{\partial t} + \bm{v}  \cdot \nabla_{\bm{x}}  f &= \frac{-i}{(2 \pi)^d H^{d+1}} \int_{\mathbb{R}^d_{\bm{v} '}}\int_{\mathbb{R}^d_{\bm{x} '}}
d\bm{v} ' d\bm{x} ' \, e^{\left( i \frac{\bm{v} ' - \bm{v} }{H} \cdot \bm{x} ' \right) }\\
& \qquad \times \left[ \Phi\left(\bm{x}  + \frac{\bm{x} '}{2}\right) - \Phi\left(\bm{x}  - \frac{\bm{x} '}{2}\right)\right] f(\bm{x} ,\bm{v} ', t),  \nonumber \\
- \Delta_{\bm{x}}  \Phi &= \int_{\mathbb{R}^d_{\bm{v} }} f \, d\bm{v}  - 1, \label{eq:poisson}
\end{align}
\end{subequations}
An effective adaptive-rank WP solver should preserve the intrinsic physical invariants of the continuous system. We therefore derive the moment equations that define the target invariants for the later correction step in Section \ref{subsec:global_conservation}. 
The WP system possesses several physical invariants. To facilitate their definitions, let us define the integrals over the spatial domain $\mathbb{R}^d_{\bm{x}}$ and the velocity domain $\mathbb{R}^d_{\bm{v}}$:
$$ \Big\langle \cdot \Big\rangle_{\mathbb{R}^d_{\bm{x}}} := \int_{\mathbb{R}^d_{\bm{x}}} (\cdot) \,d\bm{x}, \quad \Big\langle \cdot \Big\rangle_{\mathbb{R}^d_{\bm{v}}} := \int_{\mathbb{R}^d_{\bm{v}}} (\cdot) \,d\bm{v}, \quad \Big\langle \cdot \Big\rangle_{\mathbb{R}^d} := \int_{\mathbb{R}^d_{\bm{x}}}\int_{\mathbb{R}^d_{\bm{v}}} (\cdot) \,d\bm{v}\,d\bm{x}. $$
The macroscopic moments, namely density $\rho(\bm{x},t)$, current density $\bm{J}(\bm{x},t)$, pressure $P(\bm{x},t)$, and heat flux $\bm{Q}(\bm{x},t)$, are then defined as:
$$ \rho(\bm{x},t) = \Big\langle f \Big\rangle_{\mathbb{R}^d_{\bm{v}}}, \, \bm{J}(\bm{x},t) = \Big\langle \bm{v} f \Big\rangle_{\mathbb{R}^d_{\bm{v}}}, \,  P(\bm{x},t) = \Big\langle |\bm{v}|^2 f \Big\rangle_{\mathbb{R}^d_{\bm{v}}}, \, \bm{Q}(\bm{x},t) = \Big\langle |\bm{v}|^2 \bm{v} f \Big\rangle_{\mathbb{R}^d_{\bm{v}}} $$
On periodic domains, the WP system preserves the total mass $\mathcal{M}[f] = \langle \rho \rangle_{\mathbb{R}^d_{\bm{x}}}$, total momentum $\mathcal{J}[f] = \langle \bm{J}\rangle_{\mathbb{R}^d_{\bm{x}}}$, and total energy $\mathcal{E}[f] = \frac{1}{2}\langle P\rangle_{\mathbb{R}^d_{\bm{x}}} + \frac{1}{2}\langle |\nabla \Phi|^2 \rangle_{\mathbb{R}^d_{\bm{x}}}$, where $\nabla \Phi$ is the electric field computed at the target time level.
\begin{theorem}
    Assume periodic boundary conditions on the spatial domain $\Omega_x$. Then the WP system satisfies the following quantum-fluid moment equations.
        \begin{subequations}
        \begin{align}
            \frac{\partial \rho}{\partial t} + \nabla \cdot \bm{J} & = 0 \label{eqn:quantum_fluid_model_a}\\
            \frac{\partial \bm{J}_i }{\partial t} + \sum_{j=1}^d \frac{\partial}{\partial x_j} \left( \int_{\mathbb{R}^d} v_i v_j f \, dv \right) & = -\rho (\frac{\partial}{\partial x_i}\Phi), i = 1,\cdots,d \label{eqn:quantum_fluid_model_b}\\
           \frac{\partial P}{\partial t} + \nabla \cdot \bm{Q} &= -2 \boldsymbol{J} \cdot \nabla \Phi.\label{eqn:quantum_fluid_model_c}
        \end{align}\label{eqn:quantum_fluid_model}
    \end{subequations}
Furthermore, the conservation laws of mass, momentum, and energy 
can be derived through these equations. 
\end{theorem}
\begin{proof}
    Multiplying \eqref{eq:wigner} by $1, \bm{v}, |\bm{v}|^2$, and integrating over velocity $\mathbb{R}_v^d$ gives \eqref{eqn:quantum_fluid_model_a} -- \eqref{eqn:quantum_fluid_model_c}. Integrating over the periodic spatial domain $\Omega_x$ cancels divergence terms and we have 
    \[
    \partial_t \mathcal{M}[f](t) = 0, \quad \partial_t \mathcal{J}[f](t) = 0, \quad \partial_t \mathcal{E}[f](t) = 0, 
    \]
    and thus mass, momentum and energy are conserved.
    \end{proof}

\subsection{A full-rank solver based on second-order Strang splitting}\label{subsec:SSmethod}
Building on the nondimensionalized form \cref{eqn:nonD-WP-system}, we introduce the improved full-rank solver based on the Strang splitting \cite{arnold1995operator}. The high-dimensional extension of the full-rank solver is a direct generalization of the 1D1V case \cite{christlieb2025sampling}. For completeness, we summarize the core ideas in this section. In the following sections, unless otherwise specified, 
the periodic spatial domain $\Omega_{\bm{x}}$ is discretized along each dimension with a uniform Cartesian mesh,
\begin{equation}\label{eqn:spatial_dom_dis}
    a^{(\mu)} = x^{(\mu)}_{0} < x^{(\mu)}_{1} < \ldots < x^{(\mu)}_{N^{\mu}_{\bm{x}}} = b^{(\mu)} - \Delta x^{(\mu)}, 
    \quad \mu=1,2,\ldots,d,
\end{equation}
with $\Delta x^{(\mu)} = x^{(\mu)}_{1} - x^{(\mu)}_{0}$ and $L_x^{(\mu)} = b^{(\mu)}  - a^{(\mu)} $. Similarly, we define a uniform Cartesian mesh on a symmetric velocity domain $\Omega_{\bm{v}}$,
\begin{equation}\label{eqn:velocity_dom_dis}
    -L_v^{(\mu)} = v^{(\mu)}_{0} < v^{(\mu)}_{1} < \ldots < v^{(\mu)}_{N^{\mu}_{\bm{v}}} = L_v^{(\mu)} - \Delta v^{(\mu)}, 
    \quad \mu=1,2,\ldots,d,
\end{equation}
where $\Delta v^{(\mu)} = v^{(\mu)}_{1} - v^{(\mu)}_{0}$, and $L_v^{(\mu)}$ serves as the cutoff bound of the original infinite velocity domain. Let $f^{l}_{i_1,\ldots,i_d,j_1,\ldots,j_d}$ denote the numerical solution at $(\bm{x}_{i_1,\ldots,i_d},\bm{v}_{j_1,\ldots,j_d})$, where $\bm{x}_{i_1,\ldots,i_d} = (x_{i_1}^{(1)},\ldots,x_{i_d}^{(d)})$, $\bm{v}_{j_1,\ldots,j_d} = (v_{j_1}^{(1)},\ldots,v_{j_d}^{(d)})$, the index $l \in \{(*,n+1/2), (*,n+1),n,n+1\}$ represents different time levels and {$\Phi^{n+1/2}_{i_1,\ldots,i_d}$ represents the potential approximation at $(\bm{x}_{i_1,\ldots,i_d},(n+1/2)\Delta t)$.}

The foundation of our approach is a second-order Strang splitting solver. Given a solution $f({\bm{x}},{\bm{v}},n\Delta t)$, the method advances the system to the next time step $(n+1)\Delta t$ through the following sequence:

{\bf a. First half-step advection:} The system is advanced by $\frac{\Delta t}{2}$ in physical space, solving the advection subproblem:
\begin{align}\label{eq:split_adv1} \frac{\partial f}{\partial t} + {\bm{v}} \cdot \nabla_{\bm{x}} f &= 0. \end{align}
This yields an intermediate solution, $f^{*,n+1/2}$. The method used here is the conservative semi-Lagrangian (SL) method with high-order WENO reconstruction, which is the same as in the 1D1V WP solver \cite{christlieb2025sampling}. The SL framework solves the advection equation by tracing characteristics backward in time, because the solution is constant along the characteristic curves. For a fixed velocity $\bm{v}_j$, the solution at the intermediate step $t^{n+1/2}$ is found by tracing back to the departure point $\bm{x}_\star = (x^{(1)}_{\star},\ldots,x^{(d)}_{\star})$, where, in particular, $x_{\star}^{(\mu)} = \mod\left(x_{i_{\mu}^{(\mu)}} - v_{j_{\mu}^{(\mu)}}  (\frac{\Delta t}{2}), L_x^{(\mu)}\right)$, such that \begin{equation}
    f_{j_1,\ldots,j_d}^{*,n+1/2}(\bm{x}_{i_1,\ldots,i_d}) = f_{j_1,\ldots,j_d}(\bm{x}_\star, t^n).
    \label{eq:SL_basic_formulation}
\end{equation}
To evaluate $f(\bm{x}_\star, t^n)$, various interpolation or reconstruction strategies can be employed. We use fifth-order conservative WENO reconstruction along each axis; details are given in \cite{christlieb2025sampling,zheng2025semi,qiu2010conservative}; the fifth-order spatial accuracy significantly mitigates numerical diffusion with relatively coarse computational meshes \cite{gucclu2014arbitrarily}. Implementation details are given in Appendix \ref{apdix:WENO}.

{\bf b. Poisson solver:} The density $\int_{\mathbb{R}^d_{{\bm{v}}}} d{\bm{v}}\, f^{*,n+1/2}$ is computed, and the electrostatic potential is updated by solving the following equation using the Fast Fourier Transform (FFT) due to the periodic boundary conditions.
\begin{align}\label{eq:split_poisson} 
-\Delta_{\bm{x}} \Phi = \int_{\mathbb{R}^d_{{\bm{v}}}} d{\bm{v}}\, f^{*,n+1/2} - 1 := \bm{\rho}^{*,n+1/2} - 1. 
\end{align}
For the right-hand side of the Poisson equation, the integration over velocity is computed using the trapezoidal rule on a symmetric, finite velocity domain. Since $f(\bm{x},\bm{v})$ decays exponentially as $|\bm{v}| \to \infty$, the quadrature error is primarily determined by the domain truncation~\cite[Chap.~17]{boyd2001chebyshev}. In implementation, both integrals reduce to simple discrete summations multiplied by $\Delta v$ and $\Delta x$, respectively.

Let $\widehat{\Phi}$ denote the FFT of $\Phi$, and let $k_{x^{(1)},i_1}, k_{x^{(2)},i_2}, \cdots, k_{x^{(d)},i_d}$ be the discrete wavenumbers associated with the domain lengths and grid resolution. The FFT diagonalizes the Poisson equation~\cref{eq:poisson}, yielding for each nonzero mode $(i_1,i_2,\cdots,i_d)$:
\begin{equation}
    \label{eq:FFT Poisson solution}
    \widehat{\Phi}_{i_1,i_2,\cdots,i_d} = \frac{1}{\sum_{\mu = 1}^d k_{x^{(\mu)},i_{\mu}}^2} \, \widehat{\rho}_{i_1,i_2,\cdots,i_d}, 
    \quad (i_1,i_2,\cdots,i_d) \neq (0,0,\cdots,0).
\end{equation}
We enforce $\widehat{\Phi}_{0,0,\cdots,0} = 0$ to ensure a zero-mean solution and apply an inverse FFT back to $\Phi^{n+1/2}$ needed for the Wigner term. 

{\bf c. Full-step Wigner operator:} Under the frozen-field approximation, in which $\Phi$ is held fixed over the time interval $\Delta t$, a full $\Delta t$ step is taken on the Wigner operator using $f^{*,n+1/2}$ as the initial condition:
\begin{align}\label{eq:split_dbintegral} 
\frac{\partial f}{\partial t} = -\frac{i}{(2 \pi)^d H^{d+1}} \int_{\mathbb{R}^d_{{\bm{v}}'}}\int_{\mathbb{R}^d_{{\bm{x}}'}} d{\bm{v}}' d{\bm{x}}' \, & \exp\left( i \frac{{\bm{v}}' - {\bm{v}}}{H} \cdot {\bm{x}}' \right) \nonumber \\
& \times \left[ \Phi\left({\bm{x}} + \frac{{\bm{x}}'}{2}\right) - \Phi\left({\bm{x}} - \frac{{\bm{x}}'}{2}\right)\right] f({\bm{x}},{\bm{v}}', t). 
\end{align}
This produces the second intermediate state, $f^{*,n+1}$, as detailed in Section \ref{subsec:FFT and Wigner}.

{\bf d. Second half-step advection:} Finally, the solution $f^{*,n+1}$ is advanced by another half-step $\frac{\Delta t}{2}$ using the advection subproblem \eqref{eq:split_adv1} to obtain the final numerical solution, $f^{n+1}$. Optional correction steps can be applied afterward to enforce conservation laws.

\subsection{Fourier update for the Wigner operator}\label{subsec:FFT and Wigner}
The Wigner update is efficient in velocity Fourier variables, but this formulation introduces a discrete Hermitian-symmetry constraint. Preserving that constraint is the first structure-preservation requirement of the WP solver.
In the subsequent subsections, we detail the Fourier approach with structure-preserving properties employed to efficiently handle the non-local Wigner potential operator.
Recall that the Fourier transform and its inverse are defined as:
$$\tilde{f}(\omega) := \mathcal{F}\{f(t)\} = \int_{\mathbb{R}^d} f(t)\, e^{-i \omega \cdot t} \, dt, \quad \mathcal{F}^{-1}\{\tilde{f}(\omega)\} = \frac{1}{(2\pi)^d} \int_{\mathbb{R}^d} \tilde{f}(\omega)\, e^{i \omega \cdot t} \, d\omega.$$
Applying the Fourier transform with respect to $\bm{v}$ to the full-step Wigner equation \eqref{eq:split_dbintegral}, and letting $\bm{k}_{\bm{v}}$ be the corresponding Fourier conjugate variable, we focus on the integral term to obtain:
\begin{align*}
    \mathcal{F}_{\bm{v}}\left\{ \int \exp\left(i \frac{{\bm{v}}' - {\bm{v}}}{H} \cdot {\bm{x}}'\right) f({\bm{x}},{\bm{v}}',t)\, d{\bm{v}}' \right\} & = \mathcal{F}_{\bm{v}}\left\{ e^{-i \frac{{\bm{v}}}{H} \cdot {\bm{x}}'} \right\} \cdot \mathcal{F}_{\bm{v}}\{f({\bm{x}},{\bm{v}},t)\} \\
    & = (2\pi)^d\, \delta\left(\bm{k}_{\bm{v}} + \frac{{\bm{x}}'}{H}\right) \tilde{f}({\bm{x}},\bm{k}_{\bm{v}},t),
\end{align*}
where $\tilde{f}$ is the Fourier transform of the solution $f({\bm{x}},{\bm{v}},t)$ in velocity space, and $\delta$ denotes the Dirac delta function. Substituting this back into the Wigner subproblem yields:
$$\frac{\partial \tilde{f}}{\partial t} = -\frac{i}{H^{d+1}} \int_{\mathbb{R}^d} d{\bm{x}}'\, \delta\left(\bm{k}_{\bm{v}} + \frac{{\bm{x}}'}{H}\right) \left[ \Phi\left({\bm{x}} + \frac{{\bm{x}}'}{2}\right) - \Phi\left({\bm{x}} - \frac{{\bm{x}}'}{2}\right) \right] \tilde{f}({\bm{x}},\bm{k}_{\bm{v}},t).$$
We now introduce the change of variables $\tilde{{\bm{x}}} = \bm{k}_{\bm{v}} + \frac{{\bm{x}}'}{H}$, which implies $d{\bm{x}}' = H^d d\tilde{{\bm{x}}}$. Applying the standard identity $\int_{\mathbb{R}^d} \delta({\bm{x}}) f({\bm{x}})\, d{\bm{x}} = f(\bm{0})$ collapses the integral, resulting in:
\begin{equation}
\label{eq:ODE_kv} 
\begin{split} \frac{\partial \tilde{f}}{\partial t} &= -\frac{i}{H} \int_{\mathbb{R}^d} d\tilde{{\bm{x}}}\, \delta(\tilde{{\bm{x}}}) \left[ \Phi\left({\bm{x}} + \frac{H(\tilde{{\bm{x}}} - \bm{k}_{\bm{v}})}{2} \right) - \Phi\left({\bm{x}} - \frac{H(\tilde{{\bm{x}}} - \bm{k}_{\bm{v}})}{2} \right) \right] \tilde{f}({\bm{x}},\bm{k}_{\bm{v}},t) \\ &= \frac{i}{H} \left[ \Phi\left({\bm{x}} + \frac{H \bm{k}_{\bm{v}}}{2} \right) - \Phi\left({\bm{x}} - \frac{H \bm{k}_{\bm{v}}}{2} \right) \right] \tilde{f}({\bm{x}},\bm{k}_{\bm{v}},t). 
\end{split} 
\end{equation}
Equation \eqref{eq:ODE_kv} is now a straightforward linear ordinary differential equation (ODE) in time. By holding $\Phi$ fixed over the full time step $\Delta t$, the equation admits an exact analytical solution:
\begin{equation}\label{eq:FFT_updateform} 
\tilde{f}(x,\bm{k}_{\bm{v}},t_0+\Delta t) = \tilde{f}({\bm{x}},\bm{k}_{\bm{v}},t_0) e^{\left( \frac{i}{H} \left[ \Phi\left({\bm{x}} + \frac{H \bm{k}_{\bm{v}}}{2} \right) - \Phi\left({\bm{x}} - \frac{H \bm{k}_{\bm{v}}}{2} \right) \right] \Delta t \right)},
\end{equation}
where $\tilde{f}({\bm{x}},\bm{k}_{\bm{v}},t_0)$ is the Fourier-transformed intermediate solution $\tilde{f}^{*,n+1/2}$ with respect to velocity space. Here,  the potential $\Phi$ must be evaluated at spatially shifted locations that depend on the frequency mode. Because these shifted points do not, in general, perfectly align with the spatial grid, a clamped cubic Hermite spline is utilized to interpolate $\Phi$ at these off-grid locations. For a given target point, the algorithm identifies the surrounding grid cell and extracts a local $4^d$ tensor patch. It then uses the function values, first derivatives, and cross-derivatives at the cell corners, estimated using FFT-Poisson solver to smoothly blend the potential surface using Hermite basis polynomials.

Equation~\eqref{eq:FFT_updateform} reveals an important structural property: the update involves multiplying the Fourier transform of a real-valued function by a conjugate symmetric complex function,
\begin{equation}\label{eqn:expo_g}
    g^x(\bm{k}_{\bm{v}}) := \exp \left( \frac{i}{H} \left[ \Phi\left(\bm{x} + \frac{H \bm{k}_{\bm{v}}}{2} \right) - \Phi\left(\bm{x} - \frac{H \bm{k}_{\bm{v}}}{2} \right) \right] \Delta t \right),
\end{equation}
which satisfies
$g^x(-\bm{k}_{\bm{v}}) = \overline{g^x(\bm{k}_{\bm{v}})}, \quad \forall \bm{k}_{\bm{v}} \in \mathbb{R}^d \text{ or } \mathbb{Z}^d.$

By standard Fourier theory, the transform of any real-valued function is conjugate symmetric, and conversely, the inverse transform of a conjugate symmetric function is guaranteed to be real. As a result, the updated solution should remain real-valued in exact arithmetic. However, preserving this structure in numerical practice requires careful handling, which we summarize as follows:
\begin{itemize}
    \item The number of mesh points in the velocity domain, ${N^{\mu}_{\bm{v}}}$, is restricted to a power of two to maximize the efficiency of the FFT \cite{knuth2014art}.
    \item To map the continuous infinite domain to the discrete grid required by the equation, the FFT frequency domain along each velocity axis should be zero-centered and normalized by a factor of $\frac{2\pi}{2L^{(\mu)}_{\bm{v}}}$, as shown in the following form and separated into two parts.
\begin{subequations}\label{eq:frequency space}
\begin{align}
    & K_{\bm{v}}^{(\mu)}  = \frac{2\pi}{2L^{(\mu)}_{\bm{v}}} \left\{-\frac{N^{(\mu)}_{\bm{v}}}{2}, -\frac{N^{(\mu)}_{\bm{v}}}{2}+1, \dots, 0, \dots, \frac{N^{(\mu)}_{\bm{v}}}{2}-1 \right\}, \quad \mu=1,2,\ldots,d;\\
& (K_{\bm{v}}^{(\mu)})^+  = \frac{2\pi}{2L^{(\mu)}_{\bm{v}}} \left\{1, \dots, \frac{N^{(\mu)}_{\bm{v}}}{2}-1 \right\}, \\
& (K_{\bm{v}}^{(\mu)})^-  = \frac{2\pi}{2L^{(\mu)}_{\bm{v}}} \left\{-\frac{N^{(\mu)}_{\bm{v}}}{2}, -\frac{N^{(\mu)}_{\bm{v}}}{2}+1, \dots, -1\right\}.
\end{align}   
\end{subequations}
Let $k^{(\mu)}_{Nyq} = \frac{2\pi}{2L^{(\mu)}_{\bm{v}}} (-\frac{N^{(\mu)}_{\bm{v}}}{2})$, which is called the Nyquist frequency.  
\item Similar to the 1D1V case, evaluating the Nyquist frequency introduces numerical artifacts because the truncation of $g^x(\bm{k}_{\bm{v}})$ breaks the Hermitian symmetry pairing at these nodes. To resolve this, we enforce a zero value for $\tilde{f}({\bm{x}},\bm{k}_{\bm{v}},t_0)$ at the following positions:
\begin{align}\label{eqn:S_def}
    S & = \{ \bm{k}_{\bm{v}} = (k_{j_1}^{(1)},\ldots,k_{j_d}^{(d)}), k^{(\mu)}_{j_{\mu}} \in \{0, k^{(\mu)}_{Nyq}\}, \mu \in \{1,\ldots,d\} | \exists \mu \in \{1,\ldots,d\}, \\
    & \qquad \quad  \text{such that } k^{(\mu)}_{j_{\mu}} = k^{(\mu)}_{Nyq}\}.  
\end{align}
As shown in Figure \ref{fig:FFT_illu}, the set $S$ corresponds to the blue regions (the left panel depicts the 1D case; see more details in \cite[Section 3.1.2]{christlieb2025sampling} and the right panel illustrates the 2D case) as an illustration.
The red region corresponds to the origin $\bm{0} = (0,\ldots,0)$, where the function value remains unchanged during the update. Furthermore, we use $S^0$ to denote $S \cup \{\bm{0}\}$. To interpret Figure \ref{fig:FFT_illu}, we define:
\begin{align}\label{eqn:AB_def}
    A_\mu & = \{ \bm{k}_{\bm{v}} = (k_{j_1}^{(1)},k_{j_2}^{(2)})| k^{(\mu)}_{j_{\mu}} = k^{(\mu)}_{Nyq}, \, k^{(\mu^c)}_{j_{2}} \in (K_{\bm{v}}^{(\mu^c)})^-\}, \\
    & \qquad \mu \in  \{1,2\}, \, \mu^c \in \{1,2\}\backslash\{\mu\}; \nonumber \\
    A_0 & = \{ \bm{k}_{\bm{v}} = (k_{j_1}^{(1)},k_{j_2}^{(2)})| k^{(\mu)}_{j_{\mu}} \in (K_{\bm{v}}^{(\mu)})^-, \mu \in \{1,2\} \}; \\ 
    C_0 & = \{ \bm{k}_{\bm{v}} = (k_{j_1}^{(1)},k_{j_2}^{(2)})| k^{(\mu)}_{j_{\mu}} \in (K_{\bm{v}}^{(\mu)})^+, \mu \in \{1,2\} \}; \\
    C_\mu & = \{ \bm{k}_{\bm{v}} = (k_{j_1}^{(1)},k_{j_2}^{(2)})| k^{(\mu)}_{j_{\mu}} = k^{(\mu)}_{Nyq}, \, k^{(\mu^c)}_{j_{2}} \in (K_{\bm{v}}^{(\mu^c)})^-\}, \\
    & \qquad \mu \in  \{1,2\}, \, \mu^c \in \{1,2\}\backslash\{\mu\}. \nonumber 
\end{align}
\end{itemize}
\begin{figure}[htbp]
    \centering
    \includegraphics[width=0.8\linewidth]{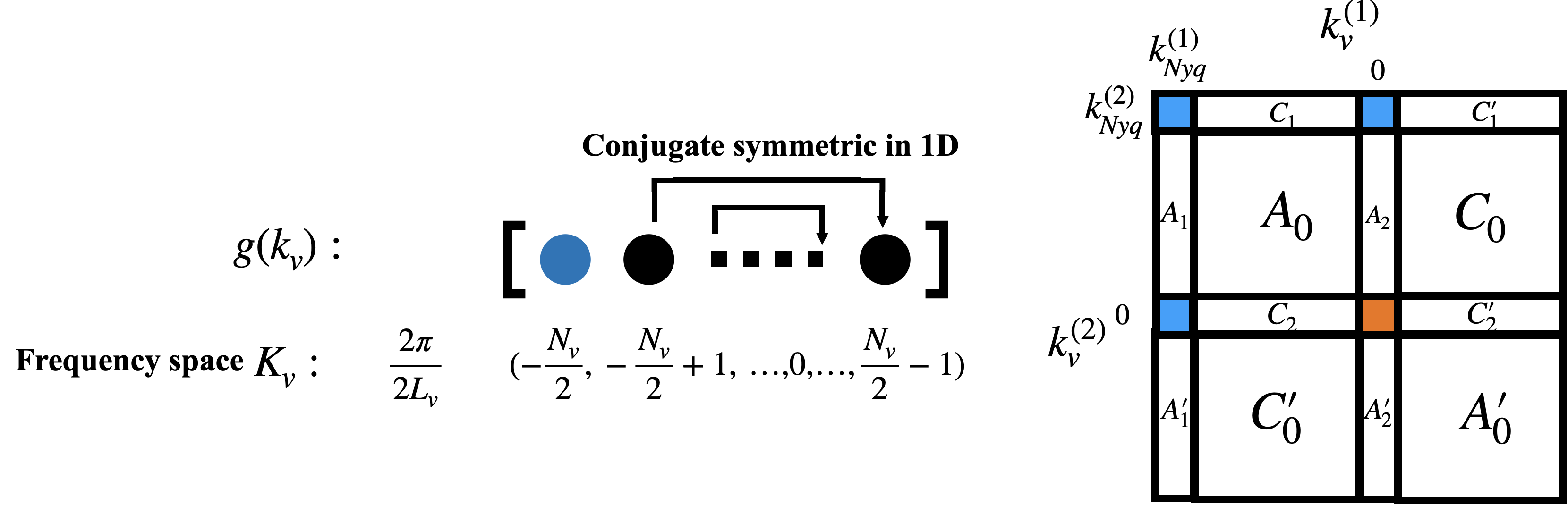}
    \caption{Discrete velocity-Fourier pairing used to enforce Hermitian symmetry with 1D on the left and 2D on the right. Blue Nyquist modes have no conjugate partners and are set to zero; the origin is retained exactly. The remaining regions are paired by index reversal and complex conjugation, which guarantees a real inverse velocity transform up to roundoff for the full-rank solver. }
    \label{fig:FFT_illu}
\end{figure}

For completeness, we present a definition of the structure-preserving property in the context of this paper. 
\begin{definition}[Structure-Preserving Scheme]\label{def:structure-preserving}
    A numerical scheme is said to be structure-preserving if it retains key analytical or physical properties of the continuous system at the discrete level. 
\end{definition}
In the present Wigner--Poisson setting, structure preservation refers to two properties. First, the velocity-Fourier update must preserve the Hermitian symmetry required for the inverse transform to remain real-valued. In particular, if $f(\bm{x},\bm{v},t_0)$ is real-valued, then the updated Fourier representation $\tilde f(\bm{x},\bm{k}_{\bm v},t_0+\Delta t)$ in \eqref{eq:FFT_updateform} must retain the Hermitian symmetry needed to recover a real-valued distribution. Second, the scheme should preserve the global invariants of the Wigner--Poisson system: mass, momentum, and self-consistent total energy. These two requirements are enforced in the sampling-based adaptive-rank algorithm developed later.

\section{Adaptive-rank Structure-preserving High-Dimensional WP Solver}\label{sec:highD_WP_solver}
Adaptive sampling and rank truncation introduce two structural difficulties in high-dimensional Wigner--Poisson simulations. First, they can destroy the Hermitian symmetry of the velocity-Fourier representation, which is needed for the reconstructed distribution function to remain real-valued. Second, they can break the global conservation of mass, momentum, and self-consistent total energy. This section develops the structure-preserving components needed to address these issues. We first recall the HTACA notation used in the adaptive-rank formulation, then introduce a Hermitian-symmetry-preserving Fourier update and a global moment correction. These ingredients are then combined into the complete adaptive-rank Wigner--Poisson solver.

Sampling-based low-rank approximation is rooted in adaptive cross approximation
and pseudoskeleton ideas, where selected rows, columns, or fibers are used to
construct compressed representations without assembling the full array
\cite{Bebendorf2000,BebendorfRjasanow2003,Goreinov1997}. Tensor extensions
include TT-cross \cite{OseledetsTyrtyshnikov2010}, black-box HT approximation
\cite{Ballani2013}, and tree-adaptive hierarchical tensor approximation
\cite{BallaniGrasedyck2014}. The present work builds on these ideas but adds
WP-specific symmetry and conservation mechanisms.

\subsection{Hierarchical Tucker Adaptive Cross Approximation}
\label{subsec:HTACA}

To reduce the cost of high-dimensional phase-space discretizations, we use the
hierarchical Tucker adaptive cross approximation (HTACA) framework of
\cite{zheng2025semihigh}. HTACA constructs a compressed hierarchical Tucker
representation from selected tensor entries, so the full tensor does not need
to be assembled. In this subsection, we introduce the notation, the tensor-tree structures used later, and the HTACA parameters needed for the symmetry-preserving Fourier update and conservation correction.

Let 
\[
\mathcal{X}\in \mathbb{C}^{n_1\times n_2\times \cdots \times n_m}
\]
denote a generic order-$m$ tensor. In the $d$-dimensional WP system considered
in this paper, the phase-space variables are
$\bm{x}\in\mathbb{R}^d$ and $\bm{v}\in\mathbb{R}^d$, so the corresponding
phase-space tensor has order $m=2d$. For example, 2D2V and 3D3V simulations
lead to order-four and order-six tensors, respectively. A hierarchical Tucker decomposition (HTD) represents this order-$m$ tensor using a
binary dimension tree $\mathcal{T}$, together with leaf frames, transfer
matrices, and hierarchical ranks $\{r_\alpha\}_{\alpha\in\mathcal{T}}.$
Each node $\alpha\in\mathcal{T}$ corresponds to a subset of tensor modes.

\begin{definition}[Dimension tree]
A binary tree $\mathcal{T}$, with each node represented by a subset of
$\{1,2,\ldots,m\}$, is called a dimension tree if it satisfies the following
conditions:
\begin{itemize}
    \item the root node is $\{1,2,\ldots,m\}$;
    \item each leaf node contains a single mode index;
    \item the two children of every non-leaf node are disjoint;
    \item each parent node is the union of its two children.
\end{itemize}
The sets of leaf and non-leaf nodes are denoted by
$\mathcal{L}(\mathcal{T})$ and
$\mathcal{N}(\mathcal{T})=\mathcal{T}\setminus\mathcal{L}(\mathcal{T})$,
respectively. For each non-leaf node
$\alpha\in\mathcal{N}(\mathcal{T})$, we denote its left and right children by
$\alpha_l$ and $\alpha_r$.
\end{definition}
To simplify notation, we assume that for each parent node, the mode indices in
the left child are smaller than those in the right child. A dimension tree for
an order-$m$ tensor contains exactly $m-1$ non-leaf nodes
\cite{grasedyck2010hierarchical_SIMAA}. \Cref{fig:dimension_trees} shows the two tree structures used in the numerical
experiments. The balanced four-mode tree is used for 2D2V simulations, while
the six-mode tree is used for 3D3V simulations. 

\begin{figure}[htbp]
    \centering
    \begin{subfigure}{0.49\textwidth}
        \centering
        \includegraphics[width=0.5\linewidth]{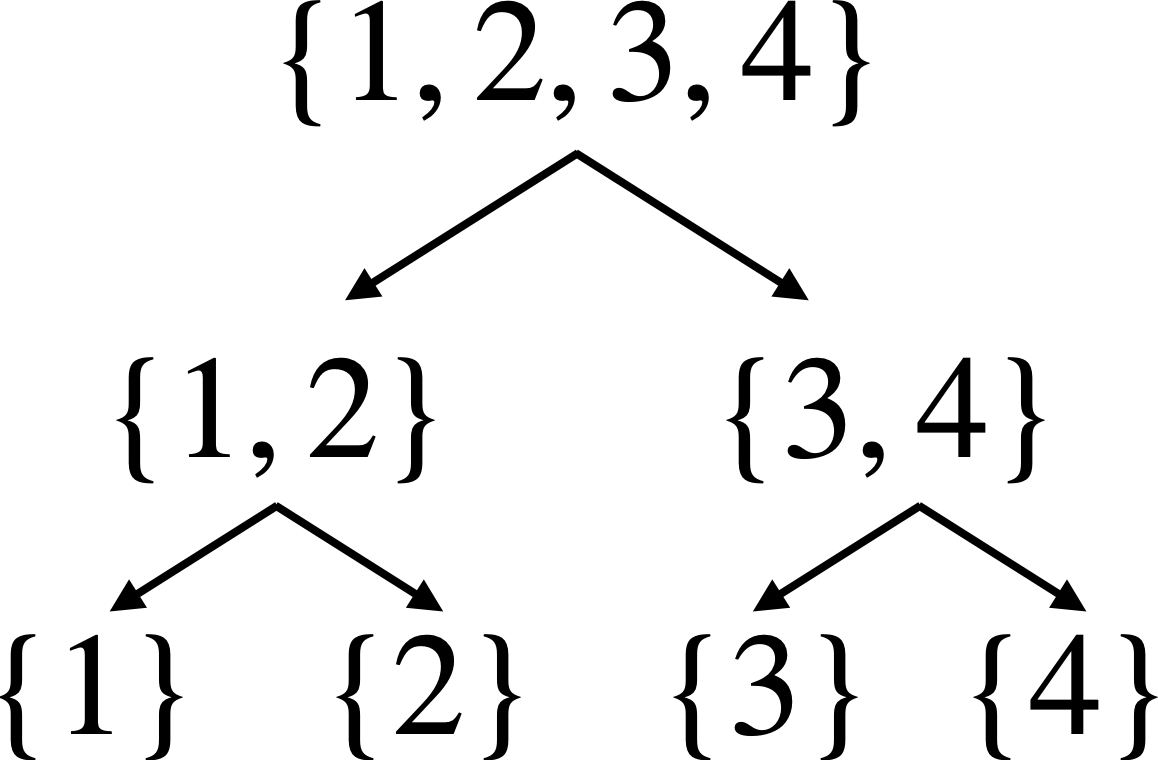}
        \caption{Balanced tree for 4D tensors.}
        \label{fig:balanced_tree_4d}
    \end{subfigure}
    \hfill
    \begin{subfigure}{0.49\textwidth}
        \centering
        \includegraphics[width=0.7\linewidth]{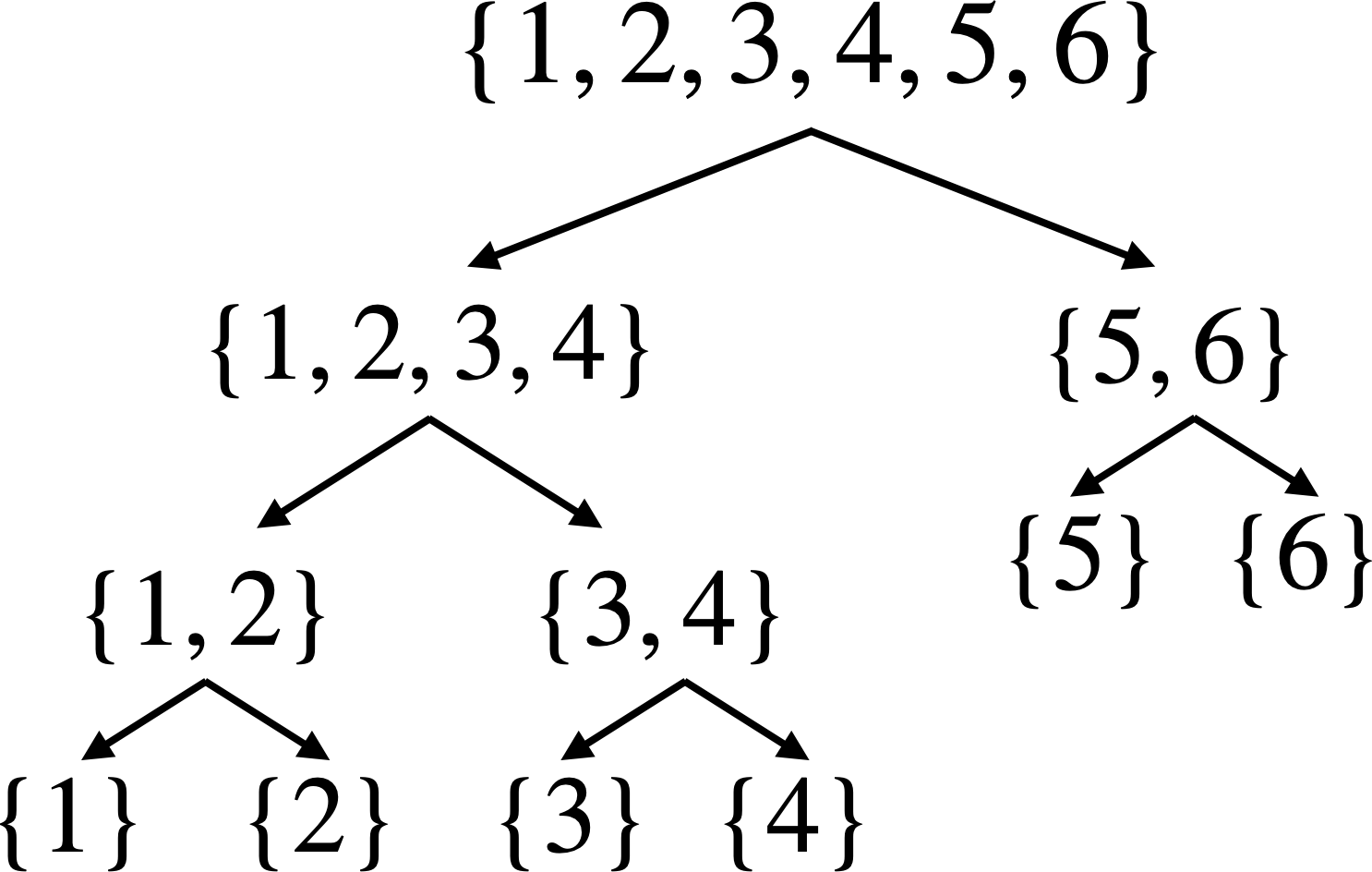}
        \caption{Unbalanced tree for 6D tensors.}
        \label{fig:unbalanced_tree_6d}
    \end{subfigure}

    \caption{Examples of hierarchical dimension trees used for the 2D2V and
    3D3V phase-space tensors. }
    \label{fig:dimension_trees}
\end{figure}

For each node $\alpha\in\mathcal{T}$, let
$\mathbb{I}_\alpha =
\prod_{\mu\in\alpha} \{1,\ldots,n_\mu\}$
denote the corresponding multi-index set. For a leaf node
$\alpha\in\mathcal{L}(\mathcal{T})$, the HTD stores a leaf frame
$U_\alpha \in \mathbb{C}^{n_\alpha \times r_\alpha}.$
For a non-leaf node $\alpha\in\mathcal{N}(\mathcal{T})$ with children
$\alpha_l$ and $\alpha_r$, the coupling between the two child subspaces is
represented by a transfer matrix
$
B_\alpha \in
\mathbb{C}^{r_{\alpha_l}r_{\alpha_r}\times r_\alpha}.
$
The corresponding frame at node $\alpha$ is recursively defined by
\begin{equation}
\label{eq:HTD_foundation}
U_\alpha =
(U_{\alpha_r}\otimes U_{\alpha_l})B_\alpha .
\end{equation}

\begin{definition}[HTD]
A tensor
$\mathcal{X}\in
\mathcal{H}\text{-Tucker}\bigl((r_\alpha)_{\alpha\in\mathcal{T}}\bigr)$
is represented by the collection of transfer matrices and leaf frames
\[
\mathcal{X}
=
\left(
(B_\alpha)_{\alpha\in\mathcal{N}(\mathcal{T})},
(U_\alpha)_{\alpha\in\mathcal{L}(\mathcal{T})}
\right),
\]
satisfying the recursive relation~\eqref{eq:HTD_foundation} for all non-leaf
nodes.
\end{definition}

The HTACA algorithm constructs such an HTD representation using only selected
entries of the target tensor. In our implementation, the target tensor is not
formed explicitly. Instead, its entries are provided by a function handle
generated from the semi-Lagrangian update, the Poisson solver, or the
velocity-Fourier update. We write the adaptive compression abstractly as
\begin{equation}
\label{alg:HTACA}
\widetilde{\mathcal{X}}
=
\mathrm{HTACA}
\left(
\mathcal{X},
\mathcal{T},
\varepsilon_{\mathrm{Base}},
r_{\min},
r_{\max}
\right),
\end{equation}
where $\mathcal{T}$ is the prescribed dimension tree,
$\varepsilon_{\mathrm{Base}}$ is the baseline tolerance used for pivot
selection and singular-value truncation, and $r_{\min}$ and $r_{\max}$ are the
lower and upper admissible hierarchical ranks.

In the semi-Lagrangian steps,
HTACA approximates the tensor generated by tracing characteristics and applying
the WENO reconstruction formula. In the Fourier step, HTACA approximates the
tensor generated by the velocity-Fourier update. In the conservation correction,
the HTD structure is used to extract the leading rank-one mode and to evaluate
moments by tensor quadrature.

The storage cost of an HTD representation is substantially smaller than that of
the full tensor when the hierarchical ranks remain moderate. If all tensor
modes have size at most $N$ and all hierarchical ranks are at most $r$,
then the storage scales as
\[
\mathcal{O}(m N r + m r^3),
\]
instead of the full tensor cost $\mathcal{O}(N^m)$. For a balanced dimension
tree with uniform mode size $N$ and uniform rank $r$, the HTACA compression step
may request up to
\[
\mathcal{O}\bigl(m N r^{\lceil \log_2 m\rceil}\bigr)
\]
tensor entries in the worst case. In the WP solver, each requested entry is
evaluated through the corresponding update formula, so the total cost depends
on both the number of sampled entries and the cost of each local reconstruction.
When the ranks remain bounded, the observed cost in the numerical examples
scales nearly linearly with respect to $N$.

Finally, we emphasize that HTD tensor addition requires all operands to share
the same dimension tree. This restriction is important for the
symmetry-preserving construction in the next subsection, where several masked
velocity-frequency sub-region tensors are generated separately and then added in
HTD form. For this reason, all masked tensors introduced below are embedded in
the same global tensor dimensions and represented using the same tree
structure.

\subsection{Hermitian-symmetry-preserving scheme}\label{subsec:structure-preserving}
A standard HTACA approximation does not automatically preserve the Hermitian
symmetry of the velocity-Fourier representation. This is problematic for the WP Fourier
update, since real-valuedness of $f(\bm{x},\bm{v},t)$ requires
\[
\hat f(\bm{x},\bm{k}_{\bm{v}},t)
=
\overline{\hat f(\bm{x},-\bm{k}_{\bm{v}},t)}.
\]
Symmetry defects therefore generate artificial imaginary components after the
inverse Fourier transform. In 1D1V, this property can be preserved by a
symmetric selection procedure \cite[Theorem 3.3]{christlieb2025sampling}. In
higher dimensions, however, this strategy is no longer compatible with
HTACA, because hierarchical recompression does not exactly preserve the sampled
skeleton values. We therefore approximate only a set of non-symmetric base
Fourier modes by HTACA and generate the remaining modes by index reversal and
complex conjugation; Nyquist modes are zeroed and the origin is copied exactly.

However, the underlying tree structure makes this mapping a non-trivial process.
In particular, we must consider three inherent properties of HTACA:
\begin{itemize}
    \item HTD addition strictly requires all operands to share an identical tree structure.
    \item Under its original tree structure, HTACA inherently requires the target domain to be a complete Cartesian product of the index spaces (i.e., a full matrix or regular tensor grid). While we can apply HTACA to a subset of the original matrix by zero-padding specific leaf indices, this subset must still form a regular, block-aligned Cartesian grid. It cannot be directly applied to geometrically irregular subsets, such as a matrix with omitted corners.
    \item Furthermore, HTACA is designed for low-rank scenarios where $r \ll N_{\min}$. Consequently, it becomes inapplicable to highly rectangular matrices or lower-dimensional slices (e.g., pure vectors). Because CUR-type approximations require selecting an equal number of rows and columns, the method fails if the required rank exceeds the length of the shortest dimension.
\end{itemize}
We describe the 2D version; the 3D construction is analogous. Based on the full-rank solver discussed in Section \ref{subsec:FFT and Wigner}, the non-symmetric base elements correspond to the union $\bigcup_{i=0}^2 (A_i \cup C_i) \cup S^0$. As illustrated in Figure \ref{fig:sym_illu1}, this union constitutes an irregular index domain rather than a regular Cartesian block. In particular, $A_1$, $A_2$, $C_1$, and $C_2$ are vectors, making their highly rectangular shapes incompatible with direct HTACA under a unified tree structure.
To overcome this structural limitation, at the time level $t_{n+1} = t_n+\Delta t$, we decompose the target tensor $\mathcal{X} = \tilde{f}(\bm{x},\bm{k}_{\bm{v}},t_{n+1})$ in \eqref{eq:FFT_updateform} into independent sub-regions:
$$\mathcal{X} = \mathcal{X}^{S^0} + \mathcal{X}^{A} + \mathcal{X}^{A'} + \mathcal{X}^{C} + \mathcal{X}^{C'},$$
where each component $\mathcal{X}^*$ (for $* \in \{S^0, A, A', C, C'\}$) denotes a masked tensor that shares the exact same tree structure as the global tensor $\mathcal{X}$. Specifically, $\mathcal{X}^*$ retains the exact values of $\mathcal{X}$ within the designated region $*$ and is strictly zero-padded everywhere else over the velocity space.
Rather than computing the full tensor at once, we compute only the independent ``base" elements and use HT tensor addition to superimpose the masked regions. The overall process consists of three steps:
\begin{itemize}
    \item {\bf Step I: Domain Partitioning.} We partition the velocity-frequency space $(k_{v}^{(1)}, k_{v}^{(2)})$ into structurally independent regions: $\{A_i, C_i, A_i', C_i'\}_{i = 0}^2$ and the corner positions denoted by $S^0$ (shown in blue and orange in Figure \ref{fig:sym_illu1}). 
    \item {\bf Step II: Sub-region HTACA.} We isolate two primary target sub-regions, block $C_0$, shown in dark green in Figure \ref{fig:sym_illu1}, and the dark gray region covering $\{A_i, C_i\}_{i = 1}^2 \cup A_0 \cup S^0$. Because subsequent superposition requires identical tree structures, we exclusively sample the target region to satisfy the error tolerance while padding the remainder of the tensor domain with zeros. This zero-padding isolates the target region from external structural artifacts while maintaining the required global dimensions.
    \item {\bf Step III: Symmetry Mapping and Superposition.} We apply the requisite structural transformation to the evaluated block (reversing the frequency indices and applying the complex conjugate to $C_0$) to generate its exact symmetric counterpart ($C_0'$, shown in light cyan). This process is repeated for the dark gray containing $\{A_i, C_i\}_{i = 0}^2$ to generate light gray region containing $\{A_i', C_i'\}_{i = 0}^2$. Finally, we sum the independent HT tensors together, denoted by $\mathcal{X}_{\text{base}}$. Special handling is required for the $S^0$-regions to ensure they remain strictly real-valued while keeping sufficient accuracy.  Specifically, the values in the blue $S$-regions are set to zero, while the orange $\{\bm{0}\}$-region must remain unchanged, perfectly matching the values of the input function $\tilde{f}(\bm{x}, \bm{0}, t_{n})$. To manipulate these specific values while rigorously preserving the shared tree structure required for HT tensor addition, we employ a subtract-and-add masking technique:
    \[
    \mathcal{X} = \mathcal{X}_{\text{base}} -   \mathcal{X}_{\text{base}}^{S^0} + \mathcal{X}^{\bm{0}}. 
    \]
    Here, $\mathcal{X}_{\text{base}}^{S^0}$ is a masked tensor that isolates the corner $S^0$-regions (subtracted from the base to enforce the zero condition), and $\mathcal{X}^{\bm{0}}$ is a masked tensor that isolates the orange region (added to inject the exact real values of the input function). Finally, a tensor truncation is applied to control the rank of the resulting sum.
\end{itemize}

\begin{figure}[htbp]
    \centering
    \includegraphics[width=0.7\linewidth]{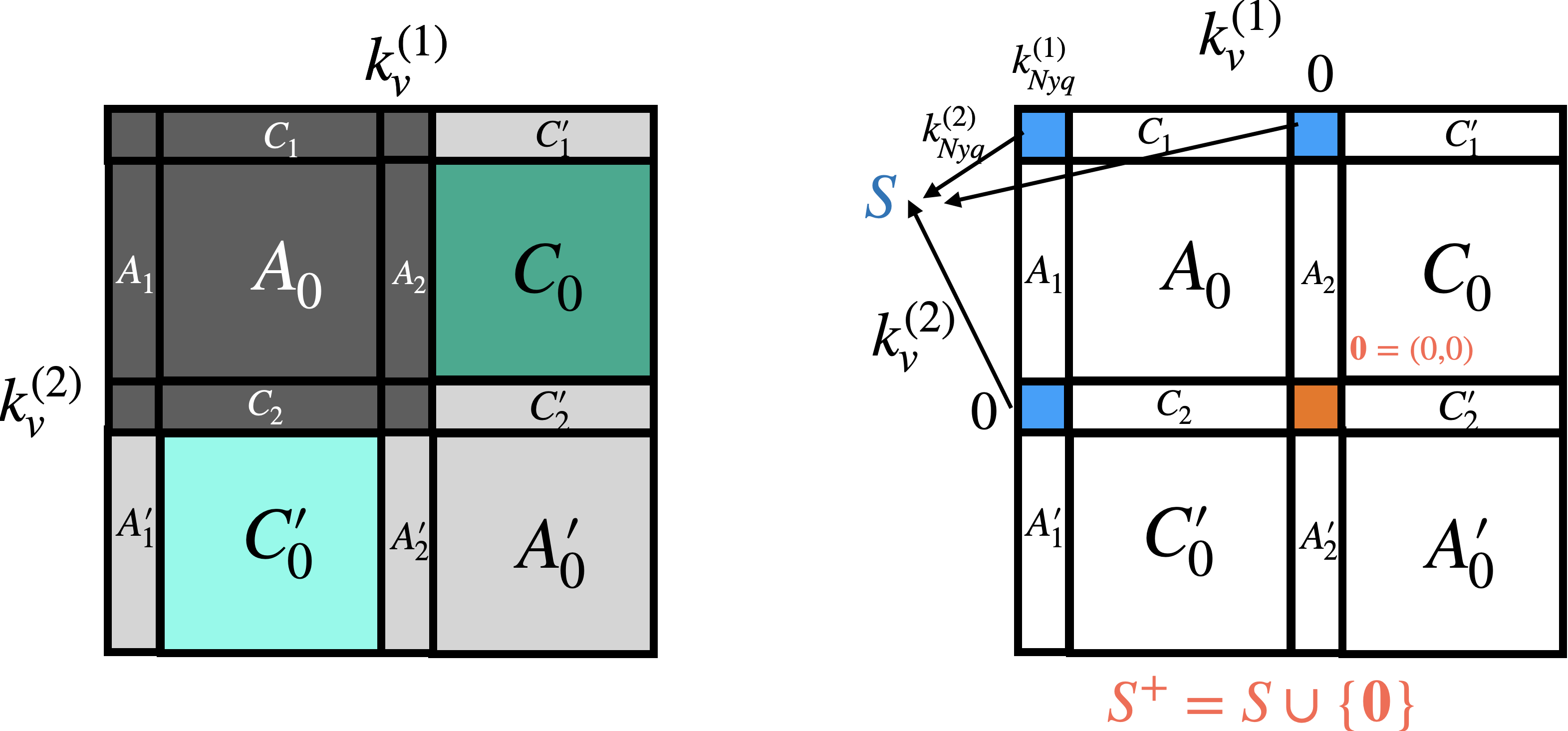}
    \caption{
    Hermitian-symmetry-preserving decomposition of the 2D velocity-Fourier plane. HTACA samples only independent base regions (dark green and dark gray); the paired regions are filled by index reversal and complex conjugation. Blue Nyquist entries are zeroed and the orange origin is restored exactly by the subtract-and-add mask, so the reconstructed Fourier tensor remains Hermitian.}
    \label{fig:sym_illu1}
\end{figure}

\subsection{Global conservation laws}\label{subsec:global_conservation}
In addition to the Hermitian symmetry property discussed in the previous section, preserving global invariants such as mass, momentum, and energy is critical for physically meaningful simulations. However, in our proposed framework, these conservation laws are generally not satisfied for two primary reasons. First, the underlying full-rank semi-Lagrangian solver does not inherently enforce global conservation. Second, even if a conservative full-rank solver were employed, the adaptive-rank procedure—which relies on truncating the solution to a lower-rank approximation—inevitably breaks these invariants. Therefore, an explicit correction step is required. To achieve this while balancing efficiency and accuracy, we propose a targeted correction procedure that restores global conservation without the computational burden of explicitly solving the macroscopic quantum fluid model \eqref{eqn:quantum_fluid_model} for $\rho$, $\bm{J}$, and $P$.

The core of our approach is to isolate the principal mode of the solution. Specifically, we utilize the first rank-one approximation of the kinetic solution, denoted by $f_1$, as the basis for this correction. There are several compelling reasons for this choice. First, $f_1$ resides within the same HTA tree structure as the uncorrected solution $f$ and naturally captures its dominant structural information. By relying on this principal mode, we ensure that the required correction is minimal, thereby preserving the robustness of the original approximation. Furthermore, this choice yields highly desirable positivity properties. For models such as the VP system, where the exact distribution is strictly non-negative, the Perron-Frobenius theorem guarantees that the leading singular mode of an all-positive matrix is also non-negative (1D1V VP). This theoretical guarantee does not strictly apply to the WP system, since the Wigner function inherently permits negative values. However, in light of this theorem, our numerical experiments consistently demonstrate that the macroscopic density remains positive for both $f$ and its principal mode $f_1$; see Table \ref{tab:min-rho-over-time}.

\begin{table}[htbp]
\centering
\caption{Minimum values of the macroscopic density $\rho$ over the final time $T = 30$ for both $f$ and its principal mode $f_1$ for 2D2V case. All simulations use the same numerical setup described in Section \ref{sec:numerical}.}
\label{tab:min-rho-over-time}
\begin{tabular}{c c c c}
\hline
Example & $H$ & $\min_{0 \le t \le 30} \rho$ for $f$ & $\min_{0 \le t \le 30} \rho$ for $f_1$ \\
\hline
TSI      & 8 & 0.0119 & 0.0224 \\
TSI       & 1 & 0.0366& 0.0119 \\
SLD      & 8 & 0.3083 & 0.0038 \\
SLD      & 1 & 0.3130 & 0.0041 \\
\hline
\end{tabular}
\end{table}

Our goal is to determine a coefficient vector
$c=(c_0,c_1,\ldots,c_d,c_E)^T$, which defines a correction polynomial restricted to the principal rank-one mode,
\begin{equation}\label{eqn:conservative_qc}
    q_c(\bm{v})
=
c_0+\sum_{i=1}^d c_i v_i+c_E|\bm{v}|^2,
\end{equation}
such that the corrected distribution function
\[f_c(\bm{x},\bm{v})
=
f(\bm{x},\bm{v})
+
f_1(\bm{x},\bm{v})q_c(\bm{v})
\]
restores the prescribed discrete invariants
\[
\mathcal{M}^\star,\quad
\mathcal{J}_1^\star,\ldots,\mathcal{J}_d^\star,\quad
\mathcal{E}^\star .
\]
These target values are the initial discrete mass, momentum, and total energy computed from the given initial data using the same quadrature rule. The coefficients are determined by imposing the conservation constraints on $f_c$. Since the correction is linear in $c$, this yields the linear system
\begin{equation}\label{eqn:conservation_linear}
    A c = \bm{R}.
\end{equation}
Here $A\in\mathbb{R}^{(d+2)\times(d+2)}$ is the moment matrix associated with $f_1$, given by
\[
A=
\begin{bmatrix}
\langle f_1\rangle_{\Omega}
&
\langle v_1 f_1\rangle_{\Omega}
&
\cdots
&
\langle v_d f_1\rangle_{\Omega}
&
\langle |\bm{v}|^2 f_1\rangle_{\Omega}
\\
\langle v_1 f_1\rangle_{\Omega}
&
\langle v_1^2 f_1\rangle_{\Omega}
&
\cdots
&
\langle v_1v_d f_1\rangle_{\Omega}
&
\langle v_1|\bm{v}|^2 f_1\rangle_{\Omega}
\\
\vdots & \vdots & \ddots & \vdots & \vdots
\\
\langle v_d f_1\rangle_{\Omega}
&
\langle v_dv_1 f_1\rangle_{\Omega}
&
\cdots
&
\langle v_d^2 f_1\rangle_{\Omega}
&
\langle v_d|\bm{v}|^2 f_1\rangle_{\Omega}
\\
\frac{1}{2}\langle |\bm{v}|^2 f_1\rangle_{\Omega}
&
\frac{1}{2}\langle |\bm{v}|^2v_1 f_1\rangle_{\Omega}
&
\cdots
&
\frac{1}{2}\langle |\bm{v}|^2v_d f_1\rangle_{\Omega}
&
\frac{1}{2}\langle |\bm{v}|^4 f_1\rangle_{\Omega}
\end{bmatrix}.
\]
The residual vector $\bm{R}$ contains the defects of the uncorrected state $f$:
\[
\begin{aligned}
\bm{R}
&=
\left(
\mathcal{M}^\star-\mathcal{M}[f],
\mathcal{J}_1^\star-\mathcal{J}_1[f],
\ldots,
\mathcal{J}_d^\star-\mathcal{J}_d[f],
\mathcal{E}^\star-\mathcal{E}[f]
\right)^T
\\
&=
\left(
\Delta\mathcal{M},
\Delta\mathcal{J}_1,
\ldots,
\Delta\mathcal{J}_d,
\Delta\mathcal{E}
\right)^T .
\end{aligned}
\]

In the above formulation, the correction step is implemented by directly solving the linear system $A c=\bm{R}$. The total energy must be evaluated consistently with the current time level. Therefore, after the second semi-Lagrangian update, the electric potential is recomputed from the current density and used in the evaluation of $\mathcal{E}[f]$. One may also ask whether the correction itself changes the electric potential and hence the electric field energy. Under the spatially uniform LoMaC-type correction, this contribution vanishes, so the correction remains a purely linear solve; see more details in Appendix \ref{apdix:energy}.

\subsection{Adaptive-rank high-dimensional WP solver}
In this section, we now summarize the complete scheme by applying HTACA to the
improved full-rank solver and incorporating the structure-preserving techniques introduced in the previous section. Algorithm \ref{alg:adaptiverank-WP solver} gives the full procedure; below, we highlight the essential
components:
\begin{enumerate}
    \item \textbf{Step 1} applies HTACA compression to the full-rank SL WENO scheme. The function handle to access the tensor entry is defined by evaluating the WENO update formula~\eqref{eq:SL_update} at the corresponding characteristic trace-back point; see details in Section \ref{subsec:SSmethod}. The HTACA algorithm (Equation \eqref{alg:HTACA} and details in \cite{zheng2025semihigh}) is applied to obtain $\mathcal{F}^{*,n+1/2}_{\alpha} = ((B_\beta^{*,n+1/2})_{\beta \in \mathcal{N}(\mathcal{T}_\alpha)},
 (U_\beta^{*,n+1/2})_{\beta \in \mathcal{L}(\mathcal{T}_\alpha)})$ in a low-rank HTD format. 

    \item \textbf{Step 2} solves the Poisson equation using the FFT method as in Section~\ref{subsec:SSmethod}, but using HTACA compression for the solver and the trapezoid rule in HTD form with low-rank complexity. In particular, as the spatial information is encoded in the leaf-node bases of the HTD, the FFT can be applied mode-wise to these matrices. The Fourier-transformed HTD of the charge density has the form
\[\widehat{\bm{\rho}} = \bigl((B_\alpha)_{\alpha \in \mathcal{N}(\mathcal{T}_{\bm{x}})}, (\widehat{U}_\alpha)_{\alpha \in \mathcal{L}(\mathcal{T}_{\bm{x}})}\bigr),\]
where each $\widehat{U}_\alpha$ is obtained by applying a 1D FFT to the columns of $U_\alpha$.  
Since $U_\alpha \in \mathbb{C}^{N_\alpha \times r_\alpha}$, the FFT is applied along rows using batched FFTs. This mode-wise FFT achieves complexity $\mathcal{O}(d_x r N \log N)$, compared with the full-grid cost $\mathcal{O}(N^{d_x} \log N)$, and is well suited to parallelization. Furthermore, we treat \cref{eq:FFT Poisson solution} as the function handle, using HTACA to get the approximation of $\Phi^{n+1/2}$: $\mathcal{P}^{n+1/2}_{\alpha}$. 

    \item \textbf{Step 3} performs the Fourier update under the adaptive-rank framework. First, similar to \textbf{Step 2} but for the velocity leaf-nodes, the FFT is applied mode-wise to these matrices, yielding the Fourier-transformed HTD $\hat{\mathcal{F}}^{*,n+1/2}_\alpha$.
    The function handle to access the matrix is defined using  ~\eqref{eqn:FFT_tensor}. To preserve Hermitian symmetry, we apply the HTACA algorithm to the non-symmetric part and then map the result to the whole domain in a specific way to obtain the Fourier update $\hat{\mathcal{F}}^{*,n+1}_\alpha$; see details in Section \ref{subsec:FFT and Wigner}. 
    After we construct $\hat{\mathcal{F}}^{*,n+1}_\alpha$, the inverse FFT is applied mode-wise to return to a
    real-valued HTD representation $\mathcal{F}^{*,n+1}_\alpha$ (up to machine precision).
    \item \textbf{Step 4} applies another HTACA semi-Lagrangian update, starting from $\mathcal{F}^{*,n+1}_\alpha$ to compute the uncorrected future state $\mathcal{F}^{n+1,NC}_\alpha$. At this stage, the macroscopic quantities—such as total mass $\mathcal{M}[f]$, total momentum $\mathcal{J}[f]$, and total energy $\mathcal{E}[f]$—are computed directly from the HTD form using multi-dimensional trapezoidal quadrature weights. For example, the uncorrected mass is evaluated as:
    \[
    \mathcal{M}[f] = \sum_{i_1 \dots i_d} \sum_{j_1 \dots j_d}  W_{i_1,\ldots,i_d,j_1,\ldots,j_d} \mathcal{F}^{n+1,NC}_{i_1,\ldots,i_d,j_1,\ldots,j_d} \prod_{\mu=1}^d \left( \Delta x^{(\mu)} \Delta v^{(\mu)} \right),
    \]
    where $W_{i_1,\ldots,i_d,j_1,\ldots,j_d}$ represents the quadrature weights (equal to $1$ in the interior and appropriately scaled by factors of $1/2$ on the boundaries).
    
    \item \textbf{Step 5} applies the global conservation correction detailed in Section \ref{subsec:global_conservation} to enforce the exact preservation of mass, momentum, and self-consistent total energy. We extract the principal rank-one approximation, denoted as $\mathcal{F}_1$, directly from the HTD of the uncorrected solution $\mathcal{F}^{n+1,NC}_{\alpha}$. By comparing the computed macroscopic moments of $\mathcal{F}^{n+1,NC}_{\alpha}$ against the exact target invariants $(\mathcal{M}^\star, \mathcal{J}_1^\star, \dots, \mathcal{J}_d^\star, \mathcal{E}^\star)$, we form the residual vector $\bm{R}$. We then solve the linear system \eqref{eqn:conservation_linear}. This yields the coefficients $c$ for the velocity correction polynomial $q_c(\bm{v})$ defined in \eqref{eqn:conservative_qc}. The final conservative solution is obtained by updating the state:
    \[
    \mathcal{F}^{n+1}_{\alpha} = \mathcal{F}^{n+1,NC}_{\alpha} + \mathcal{F}_1 q_c(\bm{v}).
    \]
    Since tensor addition increases the hierarchical rank, rank control is applied at the beginning of the next time step, before the following semi-Lagrangian update; see lines 3-5 of {\bf Step 1}. No additional truncation is applied after the correction, so the state advanced from this step satisfies the prescribed discrete invariants up to roundoff.
\end{enumerate}

\begin{algorithm}
\caption{Algorithm: Adaptive-Rank WP Solver with HTACA Compression}
\label{alg:adaptiverank-WP solver}
\begin{algorithmic}[1]
\Require Perform the spatial discretization \eqref{eqn:spatial_dom_dis} and velocity discretization \eqref{eqn:velocity_dom_dis}. Recall the notations: $f^{l}_{i_1,\ldots,i_d,j_1,\ldots,j_d}$ denotes the numerical solution at $(\bm{x}_{i_1,\ldots,i_d},\bm{v}_{j_1,\ldots,j_d})$, the index $l \in \{(*,n+1/2), (*,n+1),n,n+1\}$ represents different time levels and $\Phi^{n+1/2}_{i_1,\ldots,i_d}$ represents the potential approximation at $(\bm{x}_{i_1,\ldots,i_d},(n+1/2)\Delta t)$. In this solver, $f^{l}_{i_1,\ldots,i_d,j_1,\ldots,j_d}$ and $\Phi^{n+1/2}_{i_1,\ldots,i_d}$ are represented in the tensor form $\mathcal{F}^{l}_{\alpha}$ and $\mathcal{P}^{n+1/2}_{\alpha}$ with the tree structure needed for the HTACA method \eqref{alg:HTACA}. 
\While{$t < T$} 
    \State \textbf{Step 1: Semi-Lagrangian Method}
    \If{$n \neq 0$}
    \State Apply HTD truncation to $\mathcal{F}^{n}_{\alpha}$ to control rank growth.
    \EndIf
    \State Use~ the SL-WENO step \cref{eq:SL_update} as the function handle; apply  HTACA to obtain $\mathcal{F}^{*,n+1/2}_{\alpha} = ((B_\beta^{*,n+1/2})_{\beta \in \mathcal{N}(\mathcal{T}_\alpha)},
 (U_\beta^{*,n+1/2})_{\beta \in \mathcal{L}(\mathcal{T}_\alpha)})$.
    \State \textbf{Step 2: Poisson Equation Step}
    \State Use the trapezoidal rule in the tensor setting for $\mathcal{F}^{*,n+1/2}_{\alpha}$ along velocity space to compute the source term. 
    \State The Poisson equation \eqref{eq:poisson} is solved by the FFT in \cref{eq:FFT Poisson solution}, with periodic boundary conditions, and zero-mean potential. The Fourier transform is applied to the leaf-node matrices for the spatial axes, and then \cref{eq:FFT Poisson solution} is used as the function handle to apply HTACA to obtain an approximation of the potential $\mathcal{P}^{n+1/2}_{\alpha}$. 
    \State \textbf{Step 3: Fourier-ODE Step}
    \State Apply FFT to $\mathcal{F}^{*,n+1/2}_{\alpha}$ along velocity space to get $\hat{\mathcal{F}}^{*,n+1/2}_{\alpha}$;
    \State Solve ODE in Fourier space with \eqref{eq:FFT_updateform} and \eqref{eqn:expo_g} and specifically,  
    \begin{equation}\label{eqn:FFT_tensor}
        \hat{\mathcal{F}}^{*,n+1}_{i_1,\ldots,i_d,j_1,\ldots,j_d} = \hat{\mathcal{F}}^{*,n+1/2}_{i_1,\ldots,i_d,j_1,\ldots,j_d}g^{\bm{x}_{i_1,\ldots,i_d}}(\bm{k}_{\bm{v};{j_1,\ldots,j_d}}),
    \end{equation}
    \State Here, the reconstruction of potential $\mathcal{P}^{n+1/2}_{\alpha}$ at the off-grid point is used in \eqref{eqn:expo_g}. Apply structure-preserving HTACA with \eqref{eqn:FFT_tensor} serving as the function handle. Details of the structure-preserving properties and implementation are given in Section \ref{subsec:FFT and Wigner} and Section \ref{subsec:structure-preserving}. 
    \State Apply inverse FFT column-wise to the velocity-dependent leaf-node matrices to recover real-valued $\mathcal{F}^{*,n+1}_{\alpha}$.
    \State \textbf{Step 4: Second Semi-Lagrangian Step}
    \State Same as Step 1, but using $\mathcal{F}^{*,n+1}_{\alpha}$ instead of $\mathcal{F}^{n}_{\alpha}$. This yields the uncorrected tensor $\mathcal{F}^{n+1,NC}_{\alpha}$. 
    \State Repeat {\bf Step 2} to obtain updated potential and use it to compute the uncorrected macroscopic invariants $\mathcal{M}[f]$, $\mathcal{J}_1[f], \dots, \mathcal{J}_d[f]$, and $\mathcal{E}[f]$ by integrating $\mathcal{F}^{n+1,NC}_{\alpha}$ via the trapezoid rule. 
    
    \State \textbf{Step 5: Global Conservation Correction}
    \State Extract the principal rank-one mode $\mathcal{F}_1$ from the uncorrected tensor $\mathcal{F}^{n+1,NC}_{\alpha}$.
    \State Compute the moments of $\mathcal{F}_1$ to construct the matrix $A$  and the residual vector $\bm{R}$ based on the target invariants $(\mathcal{M}^\star, \mathcal{J}^\star, \mathcal{E}^\star)$.
    \State Construct the velocity correction polynomial $q_c(\bm{v}) = c_0 + \sum_{i=1}^d c_i v_i + c_E |\bm{v}|^2$ using the solution $c$ of the linear system \eqref{eqn:conservation_linear}.
    \State Update the distribution function: $\mathcal{F}^{n+1}_{\alpha} = \mathcal{F}^{n+1,NC}_{\alpha} + \mathcal{F}_1 q_c(\bm{v})$.
    \State $t \gets t + \Delta t$
\EndWhile
\end{algorithmic}
\end{algorithm}

\section{Numerical examples}\label{sec:numerical}
The numerical experiments are designed to assess three aspects of the proposed solver: its ability to capture high-dimensional quantum phase-space structures, its structural preservation of Hermitian symmetry and global invariants, and its empirical computational scaling. We consider two benchmark problems, two-stream instability (TSI) and strong Landau damping (SLD), in both 2D2V and 3D3V phase spaces. In all examples, the electric field is computed using the FFT-based Poisson solver in Section~\ref{subsec:SSmethod}. The HTD truncation tolerance is set to $0.1\varepsilon_{\mathrm{Base}}$, where $\varepsilon_{\mathrm{Base}}$ denotes the prescribed relative tolerance for HTACA in the semi-Lagrangian and Wigner steps; see Equation~\eqref{alg:HTACA} and \cite{zheng2025semihigh}.

Unless otherwise stated, no maximum-rank $r_{\max}$ or minimum-rank $r_{\min}$ limits are imposed in our simulations. The hierarchical Tucker tree used for 2D2V simulations is shown in Figure~\ref{fig:balanced_tree_4d}, with leaf nodes ordered as
\[
(\{1\},\{2\},\{3\},\{4\})=(x,v_x,y,v_y).
\]
For 3D3V simulations, we use the tree in Figure~\ref{fig:unbalanced_tree_6d}, with leaf nodes ordered as
\[
(\{1\},\{2\},\{3\},\{4\},\{5\},\{6\})=(x,v_x,y,v_y,z,v_z).
\]
The time step is chosen according to
\begin{equation}\label{eq:time_step}
\Delta t = \mathrm{CFL}\bigg/\left(\sum_{\mu=1}^{d_x} \frac{v_{\max}}{\Delta x^{(\mu)}}\right).
\end{equation}
All numerical simulations were implemented in MATLAB R2023b. The numerical experiments test three claims: (i) the solver captures high-dimensional WP phase-space structures; (ii) the symmetry and global invariants are preserved after adaptive- rank compression; and (iii) the observed cost scales nearly linearly with $N$ when the hierarchical ranks remain moderate. For each benchmark, we specify all parameters needed to reproduce the run in Table \ref{tab:reproducibility-setup}.

\begin{table}[htbp]
\centering
\scriptsize
\setlength{\tabcolsep}{3.2pt}
\renewcommand{\arraystretch}{1.12}
\caption{
Reproducibility setup for the TSI and SLD experiments. Entries before/after
the slash indicate main-diagnostic/scaling settings when they differ. All runs
use periodic domains, trapezoidal velocity quadrature, the tree structures in
Fig.~\ref{fig:dimension_trees}, HTD truncation tolerance
$0.1\varepsilon_{\mathrm{Base}}$, and the adaptive-weight global conservation
correction unless otherwise stated.
}
\label{tab:reproducibility-setup}
\resizebox{\textwidth}{!}{%
\begin{tabular}{l c c c c}
\hline
Parameter
& 2D2V TSI
& 3D3V TSI
& 2D2V SLD
& 3D3V SLD \\
\hline
Spatial domain $\Omega_x$
& $[0,4\pi]^2$
& $[0,4\pi]^3$
& $[0,4\pi]^2$
& $[0,4\pi]^3$ \\
Velocity domain $\Omega_v$
& $[-2\pi,2\pi]^2$
& $[-2\pi,2\pi]^3$
& $[-2\pi,2\pi]^2$
& $[-2\pi,2\pi]^3$ \\
Initial condition
& \eqref{eqn:ini_TSI}
& \eqref{eqn:ini_TSI}
& \eqref{eqn:ini_SLD}
& \eqref{eqn:ini_SLD} \\
$H$
& $\{1,8\}$
& $\{1,8\}$
& $\{1,8\}$
& $\{1,8\}$ \\
$\alpha$
& --
& --
& $0.5$
& $1/3$ \\
Grid, diagnostic/scaling
& $512^4/(2^6,\ldots,2^{11})$
& $256^6/(2^6,\ldots,2^{11})$
& $512^4/(2^6,\ldots,2^{11})$
& $256^6/(2^6,\ldots,2^{11})$ \\
CFL, diagnostic/scaling
& $50/$--
& $5/$--
& $50/$--
& $5/$-- \\
$\Delta t$, diagnostic/scaling
& \eqref{eq:time_step}/$0.1$
& \eqref{eq:time_step}/$0.1$
& \eqref{eq:time_step}/$0.1$
& \eqref{eq:time_step}/$0.1$ \\
Final time $T$, diagnostic/scaling
& $30/5$
& $30/5$
& $30/5$
& $30/5$ \\
$\varepsilon_{\mathrm{Base}}$
& $10^{-2}$
& $10^{-2}$
& $10^{-3}$
& $10^{-4} / 10^{-3}$ \\
Rank bounds
& none
& none
& none
& none \\
Output/snapshot time
& $T=30$ / $T=5$
& $T=30$ / $T=5$
& $T=30$ / $T=5$
& $T=30$ / $T=5$ \\
\hline
\end{tabular}%
}
\end{table}

\subsection{Two-stream instability (TSI)}\label{subsec:TSI}
The initial condition of the TSI problem is given by:
\begin{equation}\label{eqn:ini_TSI}
    f_0(\bm{x},\bm{v}) = \prod_{i=1}^{d} \frac{v_i^2 e^{-v_i^2/2}}{2(\sqrt{2\pi})^{d}} \left(2 + \sum_{i=1}^{d} \cos\left(\frac{x_i}{2}\right)\right).
\end{equation}
Based on the parameter setup in Table~\ref{tab:reproducibility-setup}, we perform both 2D2V and 3D3V tests. The smaller CFL number used in the 3D3V runs helps mitigate rapid rank growth during long-time simulations. Although the adaptive-rank solver substantially reduces the computational cost, tighter truncation tolerances increase both the runtime and the hierarchical ranks. The parameters used here are chosen to demonstrate long-time high-dimensional simulations while resolving the main phase-space structures.

Figure~\ref{fig:main_phase_sld_tsi} summarizes representative TSI and SLD diagnostics retained in the main text. Panels~(a) and~(b) show 3D3V TSI phase-space slices at $T=30$. For $H=1$, the slice is
$(x,v_x,y,v_y,z,v_z)=(0,0,y,v_y,0,-3.02)$; for $H=8$, the slice is
$(x,v_x,y,v_y,z,v_z)=(0,0,y,v_y,0,0)$. These snapshots are chosen to illustrate the non-classical Wigner structures, including sign-changing phase-space regions and quantum-dispersive features. The larger value of $H$ produces more pronounced quantum effects in the observed slice.

\subsection{Strong Landau damping (SLD)}\label{subsec:SLD}
The SLD problem is initialized with
\begin{equation}\label{eqn:ini_SLD}
f_0(\bm{x},\bm{v})
    = \frac{1}{(2\pi)^{d/2}}
      \left(1+\alpha\sum_{\mu=1}^{d}\cos(k x^{(\mu)})\right)
      \exp\!\left(-\frac{|\bm{v}|^2}{2}\right).
\end{equation}
Simulations are based on the parameter setup in Table~\ref{tab:reproducibility-setup}. Panel~(c) of Figure~\ref{fig:main_phase_sld_tsi} shows a 3D3V SLD phase-space slice at $T=30$ with
$(x,v_x,y,v_y,z,v_z)=(0,0,y,v_y,0,0)$ and $H=1$, demonstrating that the adaptive-rank solver also resolves the high-dimensional SLD phase-space structure. Panel~(d) of Figure~\ref{fig:main_phase_sld_tsi} shows the SLD electrostatic-energy history with $H=8$ for both 2D2V and 3D3V cases. The energy initially damps and then approaches a plateau; for this larger value of $H$, the behavior is close to weak Landau damping, reflecting stronger quantum effects. Further recurrence behavior may require longer simulations with tighter tolerances and higher resolution, in which the electrostatic energy can undergo repeated rebounds before flattening. This will be left for future work with parallel C++ simulations.

\begin{figure}[t]
    \centering
    \begin{subfigure}{0.49\textwidth}
    \includegraphics[width=\linewidth]{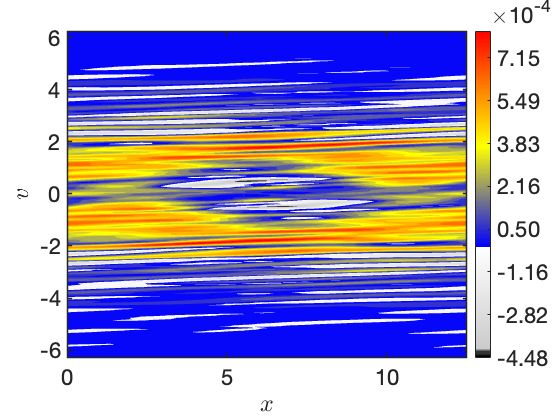}
    \caption{3D3V TSI, $H=1$, $(0,0,y,v_y,0,-3.02)$.}
    \label{fig:3D3V_phase_tsi_H1}
    \end{subfigure}
    \begin{subfigure}{0.49\textwidth}
    \includegraphics[width=\linewidth]{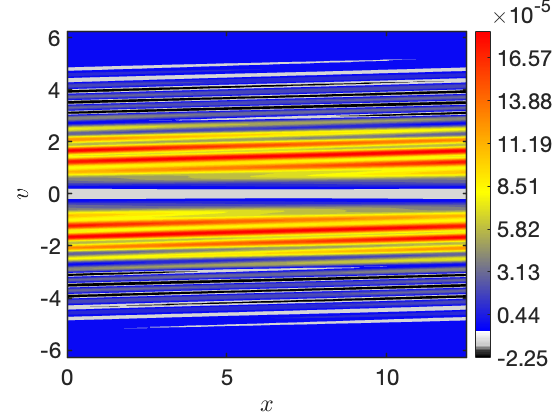}
    \caption{3D3V TSI, $H=8$, $(0,0,y,v_y,0,0)$.}
    \label{fig:3D3V_phase_tsi_H8}
    \end{subfigure}
    
    \begin{subfigure}{0.49\textwidth}
    \includegraphics[width=\linewidth]{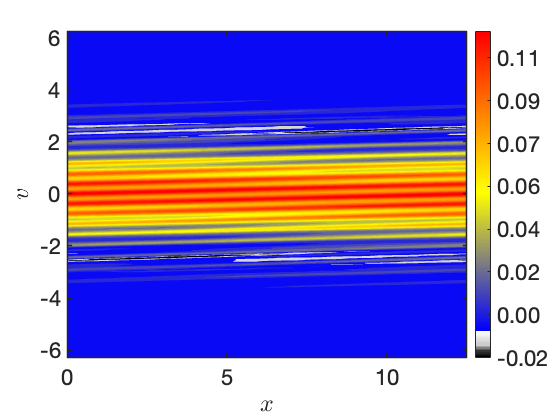}
    \caption{3D3V SLD, $H=1$, $(0,0,y,v_y,0,0)$.}
    \label{fig:3D3V_phase_sld_H1}
    \end{subfigure}
    \begin{subfigure}{0.49\textwidth}
    \includegraphics[width=\linewidth]{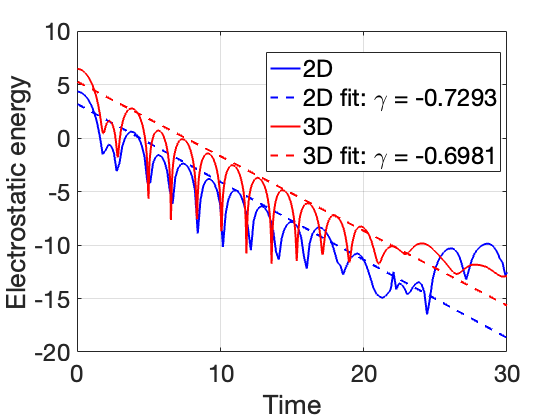}
    \caption{SLD electrostatic energy, $H=8$.}
    \label{fig:rate_sld_H8}
    \end{subfigure}
    \caption{Representative TSI and SLD diagnostics retained in the main text. Panels~(a)--(b) show 3D3V TSI phase-space slices at $T=30$ for $H=1$ and $H=8$, with slice locations $(x,v_x,y,v_y,z,v_z)=(0,0,y,v_y,0,-3.02)$ and $(0,0,y,v_y,0,0)$, respectively. Panel~(c) shows the 3D3V SLD slice at $(x,v_x,y,v_y,z,v_z)=(0,0,y,v_y,0,0)$ for $H=1$ and $T=30$. Panel~(d) shows the 2D2V and 3D3V SLD electrostatic-energy histories for $H=8$. The retained panels illustrate non-classical Wigner phase-space structures, the stronger quantum-dispersive behavior observed for larger $H$, and the SLD damping-to-plateau behavior.}
    \label{fig:main_phase_sld_tsi}
\end{figure}

\subsection{Conservation and Hermitian-symmetry preservation}
We next verify the structure-preserving properties of the proposed solver.
Table~\ref{tab:properties} summarizes the conservation and real-valuedness
diagnostics for the 2D2V and 3D3V TSI and SLD tests introduced in
Sections~\ref{subsec:TSI} and~\ref{subsec:SLD}. Instead of plotting the full
time histories, we report the maximum absolute deviation over $0\le t\le 30$.
The mass, momentum-magnitude, and total-energy defects remain near roundoff
level in all cases. The 2D2V defects are typically $10^{-13}$--$10^{-15}$,
while the 3D3V defects are about $10^{-12}$--$10^{-13}$, consistent with the
larger number of tensor operations and accumulated floating-point roundoff in
HTD contractions.

The Hermitian-symmetry diagnostic is measured by the integrated imaginary part
after transforming back to velocity space. This quantity also remains at
roundoff level, indicating that the symmetry-aware sampling and mapping
strategy prevents artificial imaginary components from accumulating after
adaptive-rank compression.

\begin{table}[t]
\centering
\caption{Maximum over time of the discrete global conservation defects and
real-valuedness diagnostic for the TSI and SLD tests over $0\le t\le 30$.
The reported quantities are $\max_t |\mathcal{M}(t)-\mathcal{M}(0)|$ for mass,
$\max_t \big||\mathcal{J}(t)|-|\mathcal{J}(0)|\big|$ for momentum magnitude,
$\max_t |\mathcal{E}(t)-\mathcal{E}(0)|$ for total energy, and
$\max_t |\text{Imaginary}(t)-\text{Imaginary}(0)|$ for the
integrated imaginary-part diagnostic. All quantities remain near roundoff level
after adaptive-rank compression, the Fourier-Hermitian-symmetry-aware update,
and the global moment correction.}
\label{tab:properties}
\begin{tabular}{lllcccc}
\hline
Dim. & Test & $H$ & $\mathcal{M}$ & $|\mathcal{J}|$ & $\mathcal{E}$ & Imaginary \\
\hline
2D2V & TSI & $1$ & $1.705{\times}10^{-13}$ & $1.231{\times}10^{-15}$ & $4.547{\times}10^{-13}$ & $6.654{\times}10^{-13}$ \\
2D2V & TSI & $8$ & $5.684{\times}10^{-14}$ & $1.394{\times}10^{-15}$ & $2.274{\times}10^{-13}$ & $1.344{\times}10^{-13}$ \\
2D2V & SLD & $1$ & $1.705{\times}10^{-13}$ & $1.023{\times}10^{-16}$ & $1.705{\times}10^{-13}$ & $5.373{\times}10^{-13}$ \\
2D2V & SLD & $8$ & $1.421{\times}10^{-13}$ & $1.453{\times}10^{-16}$ & $1.421{\times}10^{-13}$ & $1.421{\times}10^{-13}$ \\
\hline
3D3V & TSI & $1$ & $2.501{\times}10^{-12}$ & $1.435{\times}10^{-13}$ & $7.276{\times}10^{-12}$ & $2.860{\times}10^{-13}$ \\
3D3V & TSI & $8$ & $1.137{\times}10^{-12}$ & $3.793{\times}10^{-14}$ & $7.276{\times}10^{-12}$ & $3.199{\times}10^{-13}$ \\
3D3V & SLD & $1$ & $2.046{\times}10^{-12}$ & $1.354{\times}10^{-14}$ & $3.183{\times}10^{-12}$ & $2.787{\times}10^{-13}$ \\
3D3V & SLD & $8$ & $1.592{\times}10^{-12}$ & $3.129{\times}10^{-15}$ & $2.728{\times}10^{-12}$ & $2.913{\times}10^{-13}$ \\
\hline
\end{tabular}
\end{table}

\subsection{Empirical computational scaling}
Finally, we examine the empirical computational scaling of the adaptive-rank WP
solver. The tests use a fixed time step $\Delta t=0.1$, a final time $T=5$,
and the same number of grid points $N$ in each phase-space coordinate.
Figure~\ref{fig:time_fig} shows that the total CPU time grows approximately
linearly with $N$ for both 2D2V and 3D3V simulations for the prescribed
values of $\varepsilon_{\mathrm{Base}}$. In these tests,
$N=64\times 2^i$, $i=0,\ldots,5$, with
$\varepsilon_{\mathrm{Base}}=10^{-2}$ for TSI and $10^{-3}$ for SLD. This
scaling is consistent with the HTACA sampling complexity when the hierarchical
ranks remain moderate.

\begin{figure}
    \centering
    \begin{subfigure}{0.49\textwidth}
    \includegraphics[width=\linewidth]{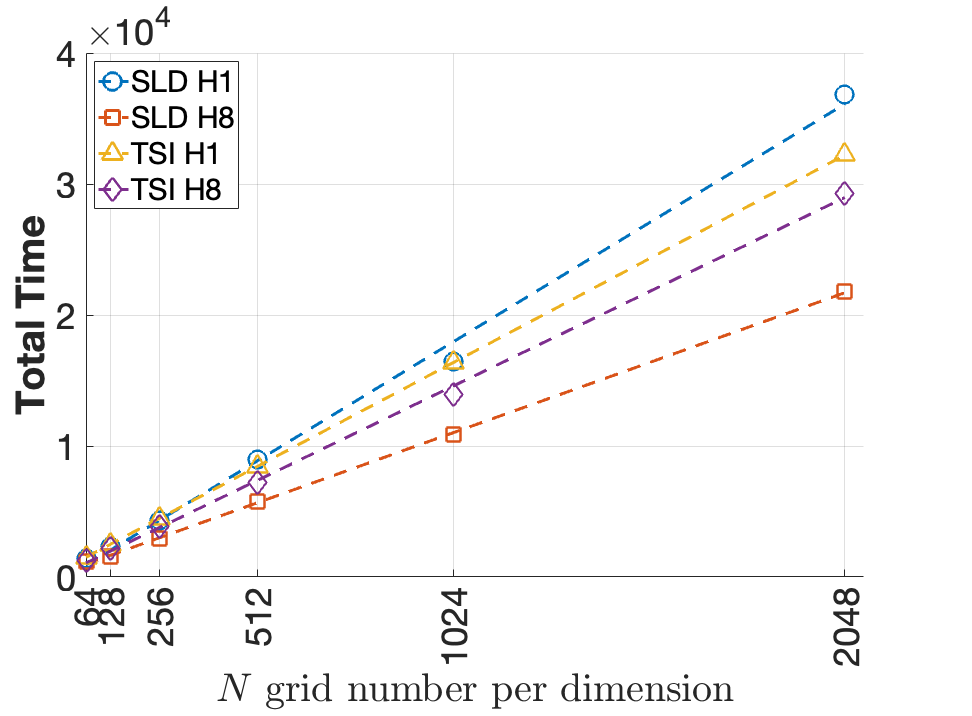}
    \caption{2D2V running time.}
    \label{fig:2D2V_time_fig}
    \end{subfigure}
    \begin{subfigure}{0.49\textwidth}
    \includegraphics[width=\linewidth]{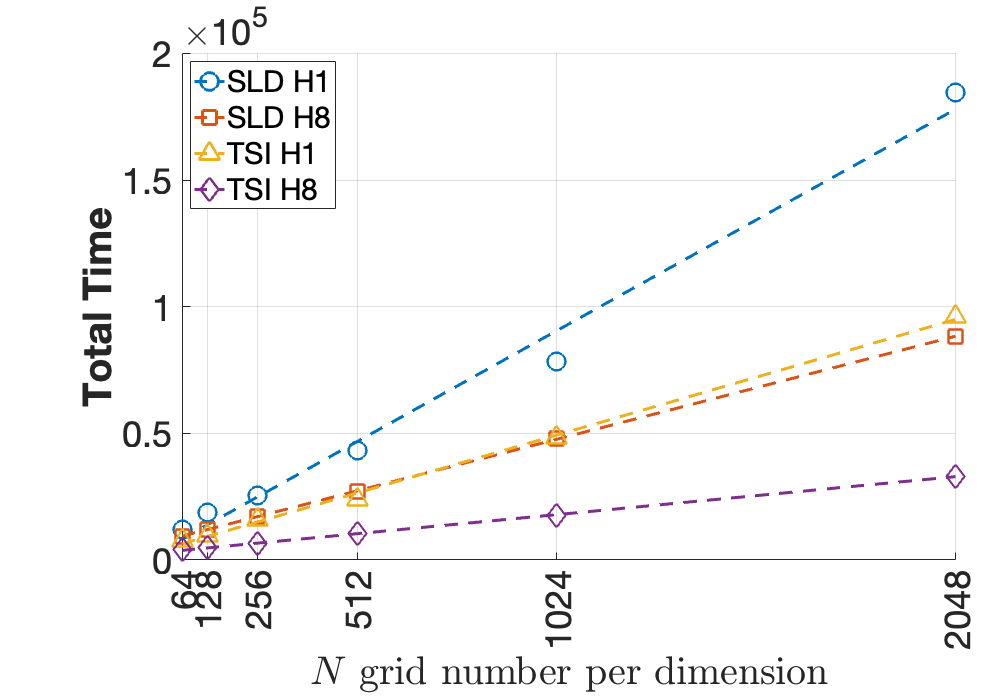}
    \caption{3D3V running time.}
    \label{fig:3D3V_time_fig}
    \end{subfigure}
    \caption{Empirical wall-clock scaling of the adaptive-rank WP solver for fixed $\Delta t=0.1$ and final time $T=5$. The total CPU running time grows approximately linearly with the number of grid points per coordinate $N$ for both 2D2V and 3D3V tests. 
    }
    \label{fig:time_fig}
\end{figure}

\section{Conclusions}\label{sec:conclusion}
We developed a structure-preserving adaptive-rank solver for high-dimensional Wigner--Poisson simulations. By exploiting low-rank structure, the proposed method reduces the computational and communication challenges caused by the nonlocal Wigner potential and mitigates the curse of dimensionality without requiring full phase-space tensor assembly. The solver combines HTACA-based sampling with two Wigner--Poisson-specific structure-preservation mechanisms: a Fourier-Hermitian-symmetry-aware velocity-Fourier update that keeps the reconstructed distribution real-valued, and a global moment correction that enforces discrete mass, momentum, and self-consistent total energy after adaptive compression. In the reported TSI and SLD tests, the method captures 3D3V phase-space dynamics, including non-classical Wigner structures such as wave-packet spreading and tunneling-related sign-changing regions, while maintaining conservation and symmetry diagnostics at roundoff level.
The timing results show approximately linear observed runtime growth over the tested problem sizes for the ranks induced by the prescribed truncation tolerances. Thus, the reported scaling should be interpreted as empirical scaling in the tested tolerance-controlled moderate-rank regimes, rather than as a fixed-rank complexity result.

The current implementation is a proof-of-concept MATLAB implementation, and the demonstrated tolerances were chosen to make long-time 3D3V runs feasible. A C++ implementation, tighter truncation tolerances, and higher resolutions are needed for more detailed studies of recurrence, fine-scale interference, and long-time quantum effects. The same structure-preserving framework also provides a natural starting point for collisional Wigner--Poisson--BGK extensions and stopping-power simulations of $\alpha$-particles.

\section*{Acknowledgements}
The authors acknowledge support from AFOSR grants FA9550-24-1-0254 and DOE grant DE-SC0023164. AJC is also supported by DOE/NNSA grant DE-NA0004265 and ONR grant N00014-24-1-2242. The authors acknowledge assistance from ChatGPT with grammar editing.

\appendix
\section{Implementation of conservative WENO quadrature}\label{apdix:WENO}
Along each axis, we find $p^{(\mu)} \in \{0,\ldots,N_x^{(\mu)}\}$, such that $x_{\star}^{(\mu)} \in [x_{p^{(\mu)}}^{(\mu)} - \frac{1}{2} \Delta x^{(\mu)},\, x_{p^{(\mu)}}^{(\mu)} + \frac{1}{2} \Delta x^{(\mu)})$ and $\xi^{(\mu)} = (x_{\star}^{(\mu)} - x_{p^{(\mu)}}^{(\mu)})/\Delta x^{(\mu)}) \in [-\frac{1}{2},\frac{1}{2})$. We present only the case for $\xi^{(\mu)} \in [-\frac{1}{2},0]$, as the remaining interval can be handled by a symmetric flip. We omit the notation for the velocity space in this part. Let $Q$ contain all the function values on the required stencils $\{\bm{x}_{\bm{p}+\bm{q}}\}$, where $\bm{q}$ satisfies for each index, $q^{(\mu)} \in \{-3, -2,\cdots,2\}$, we will apply $d$ successive 1D-WENO5 reconstruction to get the approximation: 
\begin{align}\label{eq:SL_update}
    f(\bm{x}_\star, t^n) \approx W_{1D} (\cdots W_{1D}(W_{1D}(Q,1),2),\cdots),d).
\end{align}
Here, $1,\cdots,d$ refer to the corresponding axes along which we approximate, and the 1D-WENO5 scheme, with the dimension index omitted, acting on a $1 \times 6$ vector $g = (g_{-3}, g_{-2},\cdots,g_{2})$ is given by:
\begin{align}
    W_{1D}(g) = g_0 - \xi \left(\hat{g}_{1/2}(\xi) - \hat{g}_{-1/2}(\xi) \right)
\end{align}
where $\hat{g}_{1/2}(\xi)$ and $\hat{g}_{-1/2}(\xi)$ are expressed as:
\begin{equation*}
    \hat{g}_{-1/2}(\xi) = 
    \begin{cases}
        (g_{-3},g_{-2},g_{-1},g_{0},g_{1}) \, \tilde{C}^L_5 \, (1, |\xi|, \xi^2, \xi^3, \xi^4)^{\mathrm{T}}, \quad \xi \in [-\frac{1}{2},0]\\
        (g_{-2},g_{-1},g_{0},g_{1},g_{2}) \, \tilde{C}^L_5 \,(1, |\xi|, \xi^2, \xi^3, \xi^4)^{\mathrm{T}}, \quad \xi \in (0,\frac{1}{2}).
    \end{cases}
\end{equation*}
 Here, the linear coefficient matrix $\tilde{C}^L_5$ is given by
\begin{equation}\label{eq:SL_linearmatrix}
    \tilde{C}^L_5 = \begin{bmatrix} \frac{1}{3} \omega_1 & 0 & -\frac{1}{24} & 0 & \frac{1}{120} \\ -\frac{7}{6} \omega_1 - \frac{1}{6} \omega_2 & -\frac{1}{24} & \frac{1}{4} & \frac{1}{24} & -\frac{1}{30} \\ \frac{11}{6} \omega_1 + \frac{5}{6} \omega_2 + \frac{1}{3} \omega_3 & \frac{5}{8} & -\frac{1}{3} & -\frac{1}{8} & \frac{1}{20} \\ \frac{1}{3} \omega_2 + \frac{5}{6} \omega_3 & -\frac{5}{8} & \frac{1}{12} & \frac{1}{8} & -\frac{1}{30} \\ -\frac{1}{6} \omega_3 & \frac{1}{24} & \frac{1}{24} & -\frac{1}{24} & \frac{1}{120} \end{bmatrix}, 
\end{equation}
where $$\omega_i = \frac{ \frac{d_i}{(\beta_i + \epsilon)^2} }{ \sum_{k=1}^{3} \frac{d_k}{(\beta_k + \epsilon)^2} },$$ with $\epsilon = 1\times 10^{-6}$ and $d_1 = \frac{1}{10}, \, d_2 = \frac{3}{5}, \, d_3 = \frac{3}{10}$. The smoothness indicators $\beta_i$ (for $i = 1, 2, 3$) are computed by
\begin{equation*}
    \beta_i = \left( \left(\begin{bmatrix}
1 & -4 & 3 & 0 & 0 \\
0 & 1 & 0 & -1 & 0 \\
0 & 0 & 3 & -4 & 1
\end{bmatrix} g^T \right)_i \right)^2 + \frac{13}{3} \left( \left(\begin{bmatrix}
1 & -2 & 1 & 0 & 0 \\
0 & 1 & -2 & 1 & 0 \\
0 & 0 & 1 & -2 & 1
\end{bmatrix}g^T \right)_i \right)^2.
\end{equation*}
See \cite{qiu2010conservative,christlieb2025sampling} for details on why this is defined this way.

\section{Self-consistent field energy}\label{apdix:energy}
A fully self-consistent treatment of the electric energy can also be considered. Since $f_1$ is rank one, we write $f_1(\bm{x},\bm{v})=h(\bm{x})V_1(\bm{v})$. The density perturbation induced by the correction is therefore $\delta\rho(\bm{x})=h(\bm{x})\eta(c)$, where
$
\eta(c)
=
\left\langle V_1(\bm{v})q_c(\bm{v})\right\rangle_{\Omega_{\bm v}}
=
\eta_0c_0+\sum_{i=1}^d\eta_i c_i+\eta_Ec_E.
$
Let $\chi(\bm{x})$ be the zero-mean periodic solution of
\[
-\Delta \chi(\bm{x})
=
h(\bm{x})
-
\frac{1}{|\Omega_x|}
\int_{\Omega_x}h(\bm{x})\,d\bm{x}.
\]
By linearity of the Poisson equation, the induced potential perturbation is $\delta\Phi(\bm{x})=\eta(c)\chi(\bm{x})$. Hence the corresponding change in electric field energy is
\[
\delta\mathcal{U}(c)
=
\eta(c)\Gamma_1
+
\frac{1}{2}\eta(c)^2\Gamma_2,
\qquad
\Gamma_1
:=
\int_{\Omega_{\bm{x}}}
\nabla\Phi\cdot\nabla\chi\,d\bm{x},
\quad
\Gamma_2
:=
\int_{\Omega_{\bm{x}}}
|\nabla\chi|^2\,d\bm{x}.
\]
Thus, for a general adaptive spatial weight $h(\bm{x})$, the self-consistent field-energy contribution can be included, but the total-energy constraint becomes nonlinear due to the quadratic dependence on $\eta(c)$.

In light of the LoMaC philosophy and computational efficiency, we instead take the spatially uniform choice $h(\bm{x})=1$. This imposes the conservation constraints directly on the global moments without an adaptive spatial weight. In this case, the right-hand side of the Poisson response equation vanishes after mean subtraction, so $\chi(\bm{x})=0$. Consequently, $\delta\Phi(\bm{x})=0$ and $\delta\mathcal{U}(c)=0$. Therefore, the correction does not change the electric potential or the electric field, and the enforcement of mass, momentum, and total energy reduces to solving the linear system $A c=\bm{R}$. This preserves the simplicity and efficiency of the LoMaC-type correction while still evaluating the total-energy defect $\Delta\mathcal{E}$ using the self-consistent potential obtained from the current density after the semi-Lagrangian update.

\bibliographystyle{siam} 
\bibliography{ref}
\end{document}